\pdfoutput=1
\documentclass[11pt,a4paper,dvipsnames]{scrartcl}

\usepackage[utf8]{inputenc}
\usepackage[backend=biber, giveninits=true, style=ieee, dashed=false, sorting = nty,citestyle=numeric,maxnames=8]{biblatex}

\renewbibmacro{in:}{}
\DeclareFieldFormat{pages}{#1}

\AtEveryBibitem{\clearfield{month}
                \clearfield{day}
                \clearfield{issn}
                \clearfield{doi}}

\usepackage[english]{babel}

\usepackage{geometry}

\usepackage[
    protrusion=true,
    activate={true,nocompatibility},
    final,
    tracking=true,
    kerning=true,
    spacing=true,
    factor=1100]{microtype}
\SetTracking{encoding={*}, shape=sc}{40}

\usepackage{amsthm}
\usepackage{amssymb}
\usepackage{mathtools}
\usepackage{bm}
\usepackage{csquotes}

\usepackage{enumitem}
\setlist[enumerate]{itemsep=0mm,parsep=2mm,topsep=5pt,label=\textit{(\alph*)}}
\setlist[itemize]{itemsep=1mm,parsep=2mm,topsep=6pt}

\usepackage{xcolor}
\newcommand\myshade{85}
\colorlet{myurlcolor}{Aquamarine}

\definecolor{mycitecolor}{HTML}{5590B4}

\usepackage{hyperref}
\hypersetup{
linkcolor  = black,
citecolor  = mycitecolor!85!black, %
urlcolor   = myurlcolor!\myshade!black,
colorlinks = true,
}
\usepackage[capitalise, noabbrev, nameinlink, nosort]{cleveref}
\crefname{equation}{}{}

\theoremstyle{definition}
\newtheorem{theorem}{Theorem}[section]

\newtheorem{lemma}[theorem]{Lemma}
\newtheorem{claim}[theorem]{Claim}

\newtheorem{proposition}[theorem]{Proposition}

\AddToHook{env/lemma/begin}{\crefalias{theorem}{lemma}}
\AddToHook{env/conjecture/begin}{\crefalias{theorem}{conjecture}}
\AddToHook{env/proposition/begin}{\crefalias{theorem}{proposition}}
\AddToHook{env/claim/begin}{\crefalias{theorem}{claim}}

\DeclareMathOperator{\cl}{cl}
\DeclareMathOperator{\rank}{rank}

\newcommand{\MM}{\mathcal{M}}
\newcommand{\NN}{\mathbb{N}}
\newcommand{\RR}{\mathbb{R}}
\newcommand{\CC}{\mathcal{C}}

\newcommand{\extendableconstant}{r_0}
\renewcommand{\deg}{d}

\newcommand{\dimensionality}[1]{d_{#1}}
\newcommand{\threshold}[1]{t_{#1}}

\newcommand{\graphicmatroid}{\mathcal{G}}
\newcommand{\evencycle}{\mathcal{E}}
\newcommand{\bicircular}{\mathcal{B}}

\usepackage{blkarray}

\title{Rigidity and reconstruction in matroids of highly connected graphs}
\author{Dániel Garamvölgyi\thanks{HUN-REN-ELTE Egerváry Research Group on Combinatorial Optimization and the MTA-ELTE Momentum Matroid Optimization Research Group, Pázmány Péter sétány 1/C, 1117 Budapest, Hungary, and the HUN-REN Alfréd Rényi Institute of Mathematics, Reáltanoda utca 13-15, 1053 Budapest, Hungary. e-mail: \texttt{daniel.garamvolgyi@ttk.elte.hu}}}
\date{}
\begin{document}

\maketitle
\begin{abstract}
    A \emph{graph matroid family} $\MM$ is a family of matroids $\MM(G)$ defined on the edge set of each finite graph $G$ in a compatible and isomorphism-invariant way. We say that $\MM$ has the \emph{Whitney property} if there is a constant $c$ such that every $c$-connected graph $G$ is uniquely determined by $\MM(G)$. Similarly, $\MM$ has the \emph{Lovász-Yemini property} if there is a constant $c$ such that for every $c$-connected graph $G$, $\MM(G)$ has maximal rank among graphs on the same number of vertices.

    We show that if $\MM$ is unbounded (that is, there is no absolute constant bounding the rank of $\MM(G)$ for every $G$), then $\MM$ has the Whitney property if and only if it has the Lovász-Yemini property. We also give a complete characterization of these properties in the bounded case. As an application, we show that if some graph matroid families have the Whitney property, then so does their union. Finally, we show that every $1$-extendable graph matroid family has the Lovász-Yemini (and thus the Whitney) property. These results unify and extend a number of earlier results about graph reconstruction from an underlying matroid.
\end{abstract}

\section{Introduction}

Matroids arise as an abstraction of linear independence in vector spaces and cycle-freeness in graphs. In addition to their intrinsic appeal, matroids have become a valuable tool in discrete mathematics and combinatorial optimization, for the same reason as many successful abstractions: they distill a fundamental concept -- in this case, the notion of linear independence -- to its essence, while discarding what is only incidental. 
Nonetheless, it is natural to wonder how much information is lost in this process of abstraction.

This question was answered by Whitney~\cite{whitney_1933} for graphs and their graphic matroids. More precisely, Whitney gave a characterization of the edge maps between two graphs that induce isomorphisms between their respective graphic matroids. It follows as a corollary to his result that $3$-connected graphs are uniquely determined by their graphic matroids, in a certain (strong) sense.
In other words, every $3$-connected graph is \emph{reconstructible} from its graphic matroid.\footnote{Let us stress that by reconstructibility, we mean only that the problem of reconstructing the graph from the matroid has a unique solution, and not that we can also find this solution efficiently. For graphic matroids, there is, in fact, an efficient algorithm that reconstructs the graph from its matroid, due to Seymour~\cite{seymour_1981}.}

Thus we have a statement of the form ``every sufficiently highly connected graph is reconstructible from $\MM(G)$,'' where in the case of Whitney's result, $\MM(G)$ denotes the graphic matroid of $G$. Similar results are known when $\MM(G)$ is the bicircular matroid of $G$~\cite{wagner_1985}, the $2$-dimensional generic rigidity matroid of $G$~\cite{jordan.kaszanitzky_2013}, or more generally, the $(k,\ell)$-count matroid of $G$ for $k \geq 1$ and $\ell \leq 2k-1$~\cite{garamvolgyi.etal_2024}. Apart from count matroids, the only known positive result of this type is for the generic $C^1_
2$-cofactor matroid of $G$~\cite{garamvolgyi.etal_2024}. 

The goal of this paper is the general investigation of reconstructibility problems for families of matroids associated to graphs. To make sense of this problem, we have to clarify what we mean by such a family of matroids. Roughly speaking, a \emph{graph matroid family} $\MM$ is a family of matroids $\MM(G)$, defined on the edge set of each finite graph $G$, that is isomorphism-invariant (i.e., if $G$ and $H$ are isomorphic, then so are $\MM(G)$ and $\MM(H)$), and compatible (i.e., if $H$ is a subgraph of $G$, then $\MM(H)$ is a restriction of $\MM(G)$). We say that a graph matroid family $\MM$ has the \emph{Whitney property (with constant $c$)} if there exists a constant $c$ such that every $c$-connected graph $G$ is uniquely determined by the corresponding matroid $\MM(G)$. (Precise definitions are given in the following sections.) Using this terminology, each result in the previous paragraph states that a given graph matroid family has the Whitney property.

It turns out that whether a graph matroid family $\MM$ has the Whitney property is closely tied to whether the matroid $\MM(G)$ has high rank for every sufficiently highly connected graph $G$. To make this connection more precise, let us say that a graph $G$ is \emph{$\MM$-rigid} if $\MM(G)$ has maximal rank among graphs on the same number of vertices. We say that $\MM$ has the \emph{Lovász-Yemini property (with constant $c$)} if there is a nonnegative integer $c$ such that every $c$-connected graph is $\MM$-rigid.

At first glance, the Whitney and the Lovász-Yemini properties may seem unrelated, and one might expect that little could be said about either in such a general setting. Nonetheless, it turns out that the two properties are, in fact, equivalent for almost all graph matroid families. Let us say that a graph matroid family $\MM$ is \emph{unbounded} if the rank of $\MM(K_n)$ is unbounded as a sequence of $n$, where $K_n$ denotes the complete graph on $n$ vertices. Our main result is the following theorem.

\begin{theorem}\label{theorem:main}
    An unbounded graph matroid family $\MM$ has the Whitney property if and only if it has the Lovász-Yemini property. 
\end{theorem}

Let us comment on the constants achieved by our theorem. It is easy to show that if an unbounded graph matroid family $\MM$ has the Whitney property with some constant, then it has the Lovász-Yemini property with the same constant (\cref{lemma:whitneytoLY}). Hence the main content of \cref{theorem:main} is the other implication. In this direction, we show that if $\MM$ has the Lovász-Yemini property with constant $c$, then it has the Whitney property with constant \[c' = \max(c,t) + \rank(\MM(K_{\max(c,t)-1})) + 2,\] where $t$ is a certain parameter of $\MM$ which we call its \emph{threshold}, and $K_n$ denotes the complete graph on $n$ vertices (\cref{theorem:mainprecise}). We do not know how tight this bound is in general. It is possible that $c' = c + k$ suffices for some constant $k$; perhaps for $k=2$, which would be best possible, as shown by Whitney's theorem for the graphic matroid.

In the bounded case, the equivalence between the Lovász-Yemini and the Whitney properties breaks down: every bounded graph matroid family has the Lovász-Yemini property (\cref{lemma:boundedLY}), but not all of them have the Whitney property. For completeness, we give a combinatorial characterization of those that do (\cref{theorem:boundedwhitney}).

As an application of \cref{theorem:main}, we prove that taking unions preserves the Lovász-Yemini and the Whitney properties.

\begin{theorem}\label{theorem:unionWhitney}
    Let $\MM$ be the union of the graph matroid families $\MM_1,\ldots,\MM_k$. 
    \begin{enumerate}
        \item If each of $\MM_1,\ldots,\MM_k$ has the Lovász-Yemini property, then so does $\MM$.
        \item If each of $\MM_1,\ldots,\MM_k$ has the Whitney property, then so does $\MM$.
    \end{enumerate}
\end{theorem}

\cref{theorem:main} further motivates the investigation of the Lovász-Yemini property.
Showing that a particular graph matroid family has the Lovász-Yemini property may be interesting and nontrivial even in cases when the rank function of the matroid $\MM(G)$ can be computed efficiently. We derive our terminology from a result of Lovász and Yemini~\cite{lovasz.yemini_1982}, who showed that every $6$-connected graph $G$ is $\mathcal{R}_2$-rigid, where $\mathcal{R}_2$ is the family of $2$-dimensional generic rigidity matroids. Thus, the graph matroid family $\mathcal{R}_2$ has the Lovász-Yemini property with constant $6$. In fact, the first theorem in this vein is due to Nash-Williams~\cite{nash-williams_1961} and Tutte~\cite{tutte_1961}, who showed that every $2k$-edge-connected graph contains $k$ edge-disjoint spanning trees. This implies that the $k$-fold union of the graphic matroid has the Lovász-Yemini property with constant $2k$. 

In a recent series of breakthroughs, Clinch, Jackson and Tanigawa~\cite{clinch.etal_2022a} proved a similar result for the generic $C^1_2$-cofactor matroid (which is conjectured to be the same as the $3$-dimensional generic rigidity matroid), and then Villányi~\cite{villanyi_2023} established analogous results for all $1$-extendable abstract rigidity matroids, a family which includes both the $d$-dimensional generic rigidity matroid and the generic $C^{d-2}_{d-1}$-cofactor matroid, for all $d$. 
Extensions to unions of rigidity matroids and to count matroids were given in~\cite{cheriyan.etal_2014,garamvolgyi.etal_2024,garamvolgyi.etal_2024a, jordan_2005}.

The proof method of Villányi is especially notable as it is the first example of a result of this type that is not based on a combinatorial characterization of the rank function of the corresponding matroid. Indeed, finding such a characterization for the generic $d$-dimensional rigidity matroid is a major open problem in the $d \geq 3$ case.

Following the ideas of Villányi, we prove a rather general sufficient condition for the Lovász-Yemini property. We first need some additional terminology. Let $\MM$ be a graph matroid family. We say that a graph $G$ is \emph{$\MM$-independent} if $\MM(G)$ is the free matroid on $E(G)$. The \emph{dimensionality} of $\MM$ is the supremum of the numbers $d$ such that the operation of adding a vertex of degree $d$ preserves $\MM$-independence. For a positive integer $d$, a \emph{$d$-dimensional edge split operation} consists of replacing an edge $uv$ of a graph $G$ with
a new vertex joined to $u$ and $v$, as well as to $d-1$ other vertices of $G$. Finally, a graph matroid family $\MM$ is \emph{$1$-extendable} if its dimensionality $d$ is finite and positive and the $d$-dimensional edge split operation preserves $\MM$-independence.

Using this terminology, we give the following sufficient condition for the Lovász-Yemini property.

\begin{theorem}\label{theorem:1extendable}
    Let $\MM$ be a graph matroid family. If $\MM$ is $1$-extendable, then it has the Lovász-Yemini (and hence the Whitney) property.
\end{theorem}

It is unclear whether the converse of \cref{theorem:1extendable} is also true. Let us note that the family of even cycle matroids is not $1$-extendable and does not have the Lovász-Yemini property. See \cref{subsection:extendable} for more details.

The proofs of \cref{theorem:main,theorem:unionWhitney,theorem:1extendable} are presented in \cref{section:main}. Before that, in \cref{section:graphmatroid}, we establish some basic properties of graph matroid families. Finally, we conclude in \cref{section:concluding} by highlighting potential directions for future research.

\section{Graph matroid families}\label{section:graphmatroid}

We start by setting some notation. Throughout the paper, graphs are understood to be finite and simple, that is, without loops and parallel edges. 
To avoid some trivialities, we also (often implicitly) assume that graphs have no isolated vertices. Given a graph $G$, $V(G)$ and $E(G)$ denote the vertex set and edge set of $G$, respectively. For a subset of edges $E_0 \subseteq E(G)$, we let $G[E_0]$ denote the subgraph of $G$ induced by $E_0$; similarly, we use $G[V_0]$ for the subgraph of $G$ induced by a set of vertices $V_0 \subseteq V(G)$. For a vertex $v \in V(G)$, $N_G(v)$ denotes the set of neighbors of $v$, $\partial_G(v)$ denotes the set of edges incident to $v$, and $d_G(v) = |\partial_G(v)|$ denotes the degree of $v$ in $G$. Finally, we let $K_n$ denote the complete graph on $n$ vertices, and similarly, given a finite set $V$, $K_V$ denotes the complete graph on vertex set $V$.

A \emph{graph matroid family} $\MM$ is a family of matroids $\MM(G)$, defined on the edge set of each\footnote{To avoid set-theoretical issues, we assume that every graph considered in the paper has its vertex set taken from some fixed countably infinite set, say, $\NN$.} (finite, simple) graph $G$ in a way that is
\begin{itemize}
    \item \emph{well-defined}: every graph isomorphism $\varphi: V(G) \to V(H)$ induces an isomorphism between $\MM(G)$ and $\MM(H)$; and
    \item \emph{compatible}: for every subgraph $H$ of $G$, $\MM(H)$ is the restriction of $\MM(G)$ to $E(H)$.
\end{itemize}
In the language of infinite matroids, a graph matroid family is equivalent to a finitary matroid on the edge set of the countable complete graph $K_{\NN}$ that is invariant under automorphisms of $K_{\NN}$.

The notion of graph matroid families is rather natural, and hence it is no surprise that essentially the same concept has been investigated in the past, under multiple names: as \emph{matroidal families} by Simões-Pereira~\cite{simoes-pereira_1975}, as \emph{symmetric $2$-matroids} by Kalai~\cite{kalai_1990}, and as \emph{graph matroid limits} by Király et al.~\cite{kiraly.etal_2013}. 
Since each of these sources has a different focus and uses different terminology, we opted to give a self-contained treatment, using language inspired by that of combinatorial rigidity theory. As we will see, despite being rather general objects, graph matroid families carry a surprising amount of structure.  

Let $\MM$ be a graph matroid family.
We define the \emph{rank function} $r$ of $\MM$ by letting $r(G)$ be the rank of $\MM(G)$, for each graph $G$. By a slight abuse of notation, we also use the notation $r(E_0) = r(G[E_0])$ for a subset of edges $E_0 \subseteq E(G)$. A graph $G$ is \emph{$\MM$-independent} if $r(G) = |E(G)|$, and \emph{$\MM$-dependent} otherwise. We say that $G$ is an \emph{$\MM$-circuit} if it is not $\MM$-independent, but $G-e$ is $\MM$-independent for every edge $e \in E(G)$. Finally, $G$ is \emph{$\MM$-rigid} if $r(G) = r(K_{V(G)})$. In other words, $\MM$-rigid graphs are the ones with maximal rank among graphs on the same vertex set.

Just as a matroid is uniquely determined by its independent sets, its circuits, or its spanning sets, a graph matroid family $\MM$ is uniquely determined by the collection of $\MM$-independent graphs, the collection of $\MM$-circuits, or the collection of $\MM$-rigid graphs. Let us also highlight that an edge $e \in E(G)$ is a bridge in $\MM(G)$ if and only if $e$ is not contained in any subgraph of $G$ that is an $\MM$-circuit.

We start our investigation of graph matroid families in the next subsection by introducing the notion of dimensionality and by establishing some basic but useful combinatorial results. In \cref{subsection:boundedness}, we examine graph matroid families of dimensionality zero in some more detail. Finally, in \cref{subsection:operations} we provide many examples of graph matroid families and discuss operations for constructing new ones from existing families.

\subsection{Dimensionality}

We say that a graph matroid family $\MM$ is \emph{trivial} if every graph is $\MM$-independent; otherwise $\MM$ is \emph{nontrivial}. For a nontrivial graph matroid family $\MM$, we define its \emph{dimensionality} $\dimensionality{\MM}$ and \emph{threshold} $\threshold{\MM}$ as \[\dimensionality{\MM} = \min \left\{d: \text{there exists an } \MM\text{-circuit with minimum degree } d + 1\right\}\] 
and
\[\threshold{\MM} = \min \left\{|V(C)| - 1 : C \text{ is an } \MM\text{-circuit with minimum degree } \dimensionality{\MM} + 1\right\}.\]
Note that we always have $\dimensionality{\MM} < \threshold{\MM}$.

As far as we are aware, these definitions are new. Nonetheless, the idea underlying the definitions is implicit in a number of papers. In particular, the following lemma, which shows that the rank function of a nontrivial graph matroid family grows linearly, is a special case of a theorem of Kalai~\cite[Theorem 3.2]{kalai_1990} (see also~\cite[Theorem 10.2]{kalai_1985}), and appears independently in~\cite[Proposition 4.69]{kiraly.etal_2013}. For completeness, we provide a proof.

\begin{lemma}\label{lemma:ranklinear} \emph{(Linearity of rank.)}
    Let $\MM$ be a nontrivial graph matroid family with rank function $r$, dimensionality $\dimensionality{}$, and threshold $\threshold{}$. We have $r(K_n) = \dimensionality{}(n - \threshold{}) + r(K_{\threshold{}})$ for every $n \geq \threshold{}$.
\end{lemma}
\begin{proof}
    Fix $n \geq \threshold{}$. It suffices to show that $r(K_{n+1}) = r(K_n) + \dimensionality{}$. Consider a vertex $v \in V(K_{n+1})$, and let $G$ be obtained from $K_{n+1}$ by deleting all but $\dimensionality{}$ edges incident to $v$. Since every $\MM$-circuit has minimum degree at least $\dimensionality{} + 1$, the edges incident to $v$ must be bridges in $\MM(G)$. It follows that $r(G) = r(K_n) + \dimensionality{}$. 
    
    Let us fix an $\MM$-circuit $C$ on $\threshold{}+1$ vertices and with minimum degree $\dimensionality{} + 1$. We have $r(G) = r(K_{n+1})$, since any edge $e \in E(K_{n+1}) - E(G)$ is contained in a copy of the $\MM$-circuit $C$ in $G+e$, and hence adding $e$ to $G$ does not increase the rank. It follows that $r(K_{n+1}) = r(K_n) + \dimensionality{}$, as desired.
\end{proof} 

The next lemma provides a useful interpretation of the definition of dimensionality.

\begin{lemma}\label{lemma:0extension} \emph{(Vertex addition.)}
    A nontrivial graph matroid family $\MM$ has dimensionality at least $\dimensionality{}$ if and only if the addition of vertices of degree $\dimensionality{}$ preserves $\MM$-independence. More generally, if $\MM$ has dimensionality $d$, $G$ is a graph, and $v \in V(G)$ is a vertex with $d_G(v) \leq \dimensionality{}$, then every edge incident to $v$ is a bridge in $\MM(G)$.
\end{lemma}

Finally, we have the following analogue of the so-called Gluing Lemma from combinatorial rigidity theory~\cite{whiteley_1996}.

\begin{lemma}\label{lemma:gluing} \emph{(Gluing lemma.)}
    Let $\MM$ be a nontrivial graph matroid family with dimensionality $\dimensionality{}$ and threshold $\threshold{}$. Let $G_1,G_2$ be graphs, and let $G = G_1 \cup G_2$. 
    \begin{enumerate}
        \item Suppose that $|V(G_1 \cap G_2)| \geq d$ and $|V(G_1)| \geq \threshold{}$. If $G_1$ and $G_2$ are both $\MM$-rigid, then so is $G$.
        \item Suppose that $|V(G_1 \cap G_2)| \geq t$ and $G_1 \cap G_2$ is $\MM$-rigid. If $G_1$ and $G_2$ are both $\MM$-independent, then so is $G$.
    \end{enumerate}
\end{lemma}
\begin{proof} Let $n_i = |V(G_i)|, i \in \{1,2\}$, and let $k = |V(G_1 \cap G_2)|$.
    \textit{(a)} Since $G_1$ and $G_2$ are $\MM$-rigid, we may assume that they are complete graphs. Since $n_1 \geq \threshold{}$, \cref{lemma:ranklinear} implies that
    \begin{align*}
    r(G) \leq r(K_{V(G)}) &= d(n_1 + n_2 - k - \threshold{}) + r(K_{\threshold{}}), \text{ and} \\
        r(G_1) &= d(n_1 - \threshold{}) + r(K_{\threshold{}}),
    \end{align*}
    where $r$ denotes the rank function of $\MM$.
    
    To show that $G$ is $\MM$-rigid, we need to show that the above inequality is satisfied with equality. This is equivalent to showing that $r(G) = r(G_1) + d(n_2 - k)$.
    Finally, this follows from \cref{lemma:0extension}: since $k \geq d$, every vertex in $V(G_2) - V(G_1)$ has at least $\dimensionality{}$ neighbors in $V(G_1 \cap G_2)$. Thus we can construct an $\MM$-independent subgraph of $G$ of size $r(G_1) + d(n_2 - k)$ by starting from a maximal $\MM$-independent subgraph of $G_1$ and adding vertices of degree $\dimensionality{}$.

    \textit{(b)} For $i \in \{1,2\}$, let $F_i$ be a minimal set of edges such that $G_i+F_i$ is $\MM$-rigid. Then $G_i+F_i$ is also $\MM$-independent, and since $G_1 \cap G_2$ is $\MM$-rigid, none of the edges in $F_i$ are induced by $V_1 \cap V_2$. 
    In particular, $F_1$ and $F_2$ are disjoint.

    Since $k \geq \threshold{} > \dimensionality{}$, part \emph{(a)} implies that $G+F_1+F_2$ is $\MM$-rigid. Hence by \cref{lemma:ranklinear} we have
    \begin{align*}r(G+F_1+F_2) &= d(n_1+n_2-k - \threshold{}) + r(K_{\threshold{}}), \\ r(G_i + F_i) &= d(n_i - \threshold{}) + r(K_{\threshold{}}), \hspace{.2em} i \in \{1,2\}, \text{ and} \\ r(G_1 \cap G_2) &= d(k - \threshold{}) + r(K_{\threshold{}}).
    \end{align*}
    A straightforward computation now gives
    \begin{align*}
        r(G+F_1+F_2) &= r(G_1+F_1)+r(G_2+F_2)-r(G_1\cap G_2) \\ &= |E(G_1+F_1)| + |E(G_2+F_2)| - |E(G_1\cap G_2)| \\ &= |E(G+F_1+F_2)|.
    \end{align*}
    This shows that $G+F_1+F_2$ is $\MM$-independent, and hence so is $G$, as claimed.
\end{proof}

\subsection{Boundedness}\label{subsection:boundedness}

We say that a graph matroid family $\MM$ with rank function $r$ is \emph{unbounded} if $r(K_n)$ is unbounded as a sequence of $n$; otherwise, we say that $\MM$ is \emph{bounded}. From \cref{lemma:ranklinear} it is apparent that a nontrivial graph matroid family $\MM$ is bounded if and only if its dimensionality is zero. If $\MM$ is bounded, then the sequence $\{r(K_n), n \geq 1\}$ (being monotone increasing and bounded) has a limit. We call this limit the \emph{rank of $\MM$}, and denote it by $r(\MM)$.

The following lemma provides a useful characterization of the rank of a bounded graph matroid family.

\begin{lemma}\label{lemma:unboundedforest}
    Let $\MM$ be a graph matroid family with rank function $r$.
    \begin{enumerate}
        \item $\MM$ is unbounded if and only if every forest is $\MM$-independent.
        \item If $\MM$ is bounded, then the graph consisting of $r(\MM) + 1$ disjoint edges is an $\MM$-circuit.
    \end{enumerate}
\end{lemma}
\begin{proof}
    \textit{(a)} If $\MM$ is unbounded, then by the definition of dimensionality, every $\MM$-circuit has minimum degree at least two, and thus every forest is $\MM$-independent. Conversely, if $\MM$ is bounded, then every forest with more than $r(\MM)$ edges must be $\MM$-dependent.
    
    \textit{(b)} For $k \geq 1$, let $G_k$ denote the graph consisting of $k$ disjoint edges. Let $k$ be the smallest positive integer for which $G_k$ is $\MM$-dependent. With this choice, $G_k$ must be an $\MM$-circuit. By the previous paragraph, we have $k \leq r(\MM) + 1$. To establish \textit{(b)}, we show that every graph $G$ on $k$ edges is $\MM$-dependent, and hence $r(\MM) \leq k - 1$.

    We prove the statement by induction on $\ell$, where $G$ has $k - \ell$ components. The only graph with $k$ edges and $k$ components (and without isolated vertices) is $G_k$, which is $\MM$-dependent by assumption. Thus, let us assume that $\ell \geq 1$. Then $G$ has a component that contains at least two edges, say $e$ and $f$. Let $u,v$ be a pair of new vertices. By the induction hypothesis, $G - e + uv$ and $G - f + uv$ are both $\MM$-dependent, and hence they contain $\MM$-circuits $C_1$ and $C_2$, respectively. If $uv \notin C_i$ for some $i \in \{1,2\}$, then $C_i$ is a subgraph of $G$, and thus $G$ is $\MM$-dependent. Otherwise, by the circuit elimination axiom, there is a circuit contained in $C_1 \cup C_2 - uv$. Since this is a subgraph of $G$, we deduce, once again, that $G$ is $\MM$-dependent. 
\end{proof}

Let us also note that a graph matroid family $\MM$ is unbounded if and only if for every graph $G$ and every induced subgraph $H \subseteq G$, $E(H)$ is a closed set in $\MM(G)$.

We will also need the following
technical lemma regarding the $\MM$-circuits of bounded graph matroid families. 
Given a bounded graph matroid family $\MM$ with rank $r(\MM)$, let us say that an $\MM$-circuit is \emph{small} if it has at most $r(\MM)$ edges.

\begin{proposition}\label{lemma:smallcircuits}
    Let $\MM$ be a bounded graph matroid family.
    If there is a small $\MM$-circuit with minimum degree one, then there exists an integer $m \geq 2$ such that the star $K_{1,m}$ is the unique (up to isomorphism) small $\MM$-circuit with at most $m$ edges and with minimum degree one.
\end{proposition}
\begin{proof}
    Let $r$ denote the rank function of $\MM$. Assuming that there is a small $\MM$-circuit with minimum degree one, let $m$ denote the minimum number of edges in such an $\MM$-circuit, and let $\mathcal{C}$ denote the set of small $\MM$-circuits with $m$ edges and with minimum degree one. Note that $m \geq 2$, since if $m=1$, i.e., if a single edge is an $\MM$-circuit, then $r(\MM) = 0$, and thus there cannot be any small $\MM$-circuits. 
    
    Suppose, for a contradiction, that $\mathcal{C}$ contains a graph that is not a star. We will show that $\mathcal{C}$ contains the graph consisting of $m$ disjoint edges. Since the members of $\mathcal{C}$ are $\MM$-circuits, this implies, by \cref{lemma:unboundedforest}(b), that $r(\MM) = m - 1$, contradicting the assumption that $m$ was the number of edges of a small $\MM$-circuit.

    We start by introducing a pair of operations. Given a graph $G$ and a pair of distinct edges $e$ and $uv$, where $v$ is a leaf vertex, the \emph{pruning} operation deletes $e$ from $G$ and adds a new vertex $v'$ and an edge $uv'$. (If the deletion of $e$ creates an isolated vertex, then we also delete this vertex from the graph.) Similarly, given a graph $G$ and a component $P$ of $G$ that is a path of length $2$, the \emph{path cutting} operation deletes $P$ from $G$ and adds two disjoint edges $uv$ and $u'v'$, where $u,v,u',v'$ are new vertices. Note that both of these operations preserve the property of having minimum degree one.
    \begin{claim}\label{claim:operations}
        $\mathcal{C}$ is closed under the pruning and path cutting operations.
    \end{claim}
    \begin{proof}
        Fix $C_1 \in \mathcal{C}$, and let $e$ and $uv$ be a pair of distinct edges of $C_1$, where $v$ is a leaf vertex. Let us also consider the isomorphic copy $C_2$ of $C_1$ obtained as $C_1-v+uv'$, where $v'$ is a new vertex. The strong circuit elimination axiom now implies that there is an $\MM$-circuit $C_0$ in $C = (C_1 \cup C_2) - e$ with $uv \in E(C_0)$. Note that $C$ is precisely the graph obtained from $C_1$ by the pruning operation on $e$ and $uv$. In particular, it is a graph on $m$ edges with minimum degree one. Since $uv \in E(C_0)$, $C_0$ also has minimum degree one, and hence by the minimality of $m$, we must have $C_0 = C$. Thus $C \in \mathcal{C}$.

        The proof for the path cutting operation is similar. Fix $C_1 \in \mathcal{C}$, and let $P = \{u,v,w\}$ be a component of $C_1$ that induces a path $\{uv,vw\}$. Let us consider the isomorphic copy $C_2$ of $C_1$ obtained as $C_1 - w + uv'$, where $v'$ is a new vertex. It follows from the strong circuit elimination axiom that there is an $\MM$-circuit $C_0$ in $C = (C_1 \cup C_2) - uv$ with $vw \in E(C_0)$. Now $C$ is precisely the graph obtained from $C_1$ by the path cutting operation on $P$; in particular, it is a graph  on $m$ edges with minimum degree one. As in the previous case, the minimality of $m$ implies that $C_0 = C$, and thus that $C \in \mathcal{C}$, as desired.
    \end{proof}
    \begin{claim}\label{claim:independentedge}
        $\mathcal{C}$ contains a graph with at least two components and such that one of the components is a single edge.
    \end{claim}
    \begin{proof}
        By the indirect assumption, there is a graph $C \in \mathcal{C}$ that is not a star. First, assume that $C$ has a single component. Since $C$ has minimum degree one, it has a leaf vertex $v$. Let $u$ be the vertex adjacent to $v$. 
        Let $v_1,\ldots,v_k$ denote the non-leaf neighbors of $u$; since $C$ is not a star, we have $k \geq 1$. By applying the pruning operation on $e_i = uv_i$ and $uv$, for each $i \in \{1,\ldots,k\}$, and using  \cref{claim:operations}, we obtain a member of $\mathcal{C}$ with at least two components.

        Hence we may assume that $C$ has at least two components. Let $u$ be the neighbor of a degree one vertex of $C$, and let $C'$ be a component of $C$ that does not contain $u$. Now by repeatedly applying the pruning operation to delete an edge from $C'$ and add a new edge incident to $u$, we can reduce $C'$ to a single edge. By \cref{claim:operations}, the resulting graph is a member of $\mathcal{C}$, as desired. 
    \end{proof}
    Now we are ready to finish the proof. Let $C$ be a member of $\mathcal{C}$ that has the maximum number $\ell$ of components consisting of a single edge. By \cref{claim:independentedge}, $\ell \geq 1$, and hence there is a component of $C$ consisting of some edge $uv$. If $\ell < m$, then $C$ also contains a component with at least two edges; let $e$ be an edge in such a component. By applying the pruning operation with $uv$ and $e$, and then the path cutting operation to the component containing $uv$, we obtain a graph $C'$ with at least $\ell+1$ components that consist of a single edge. Since $C' \in \mathcal{C}$ by \cref{claim:operations}, this contradicts the maximal choice of $C$.

    Hence $\ell = m$, which means that $C$ consists of $m$ disjoint edges. We deduce, using \cref{lemma:unboundedforest}(b), that $r(\MM) = m-1$, contradicting the 
    definition of $m$.
\end{proof}

\subsection{Examples and operations}\label{subsection:operations}

We close this section by illustrating the concepts introduced so far through several examples. We also recall the notion of an abstract rigidity matroid; as we will see, families of abstract rigidity matroids fit naturally within our framework of graph matroid families. Finally, we define and investigate the union and de-coning operations on graph matroid families.

\paragraph*{Graphic, bicircular, and even cycle matroids.}

The prototypical example of a graph matroid family is the family $\graphicmatroid$ of graphic matroids. A graph is $\graphicmatroid$-rigid if and only if it is connected, and the $\graphicmatroid$-circuits are simply the cycle graphs. Thus each $\graphicmatroid$-circuit has minimum degree two, and the smallest $\graphicmatroid$-circuit---the triangle---has three vertices. Hence $\graphicmatroid$ has dimensionality $1$ and threshold $2$.

Apart from $\graphicmatroid$, the most basic examples of graph matroid families may be the family $\bicircular$ of \emph{bicircular matroids}, and the family $\evencycle$ of \emph{even cycle matroids}. A graph $G$ is $\bicircular$-independent if and only if $G$ is a pseudoforest (i.e., each component of $G$ contains at most one cycle), and $G$ is $\bicircular$-rigid if and only if each component of $G$ contains a cycle. Similarly, $G$ is $\evencycle$-independent if and only if $G$ is a pseudoforest without cycles of even length, and $G$ is $\evencycle$-rigid if and only if each component of $G$ contains a cycle of odd length. These examples show that rigidity with respect to a graph matroid family does not, in general, imply connectivity. Both $\bicircular$ and $\evencycle$ have dimensionality $1$ and threshold $3$. 

As we will see shortly, there are infinitely many graph matroid families with dimensionality one. However, it is shown in~\cite[Theorem 5]{simoes-pereira_1975} that $\graphicmatroid,\bicircular$, and $\evencycle$ are the only graph matroid families with dimensionality one in which all circuits are connected graphs.

\paragraph*{Count matroids.}

Let $k$ and $\ell$ be a pair of integers with $k \geq 1$ and $\ell \leq 2k-1$. We say that a graph $G$ is \emph{$(k,\ell)$-sparse} if $i_G(X) \leq k|X| - \ell$ holds for every set of vertices $X \subseteq V(G)$ with $|X| \geq 2$. Here, $i_G(X)$ denotes the number of edges induced by $X$ in $G$. We define the family $\MM_{k,\ell}$ of \emph{$(k,\ell)$-count matroids} by letting a graph be $\MM_{k,\ell}$-independent if and only if it is $(k,\ell)$-sparse. It is known that $\MM_{k,\ell}$ is indeed a graph matroid family~\cite{lorea_1979}, and different choices of $k$ and $\ell$ give rise to different graph matroid families. Using this notation, we have $\graphicmatroid = \MM_{1,1}$ and $\bicircular = \MM_{1,0}$.

It is easy to see that $\MM_{k,\ell}$ has dimensionality $k$. Calculating the threshold is nontrivial; it can be deduced from the proof of~\cite[Lemma 2.3]{garamvolgyi.etal_2024} that when $\ell \geq 0$, the threshold $\threshold{\MM_{k,\ell}}$ is as follows: 
\begin{align*}
    \threshold{\MM_{k,\ell}} = 
    \begin{cases}
        2k-1 \hspace{1em} & \text{if } k < \ell \leq 2k-1, 
        \\ 2k  & \text{if } 0 < \ell \leq k,
        \\ 2k+1 & \text{if } \ell = 0.
    \end{cases}
\end{align*}

Schmidt~\cite{schmidt_1979} used the following modification of $\MM_{k,0}$ to construct uncountably many distinct graph matroid families. Let $k \geq 2$, and let $\CC$ be a collection of $3$-connected $2k$-regular graphs. We define the family $\MM_\CC$ by letting a graph be $\MM_\CC$-independent if it is $(k,0)$-sparse and does not contain any copy of a graph in $\CC$ as a subgraph.
Schmidt showed that $\MM_\CC$ is indeed a graph matroid family, and that different choices of $\CC$ give rise to different graph matroid families. It is unclear whether there are uncountably many graph matroid families with dimensionality one.

\paragraph*{Abstract rigidity matroids.}

The graph matroid families we have seen so far are defined in a purely combinatorial way. Another set of examples comes from discrete geometry, the most prominent being the family of generic rigidity matroids.

Given a graph $G$ and an embedding $p : V(G) \to \RR^d$, the \emph{rigidity matrix} $R(G,p)$ is a real matrix that has one row for each edge of $G$, $d$ columns for each vertex of $G$, and whose entries are given by \[R(G,p) = \begin{blockarray}{cccccc}
    & \cdots & u & \cdots  & v & \cdots \\
   \begin{block}{c(ccccc)}
      \vdots &  &  &  &  &  \\
     uv & 0  & p(u) - p(v) & 0  & p(v) - p(u) & 0  \\
      \vdots &  &  &  &  &  \\
     \end{block}
   \end{blockarray}.\]
The \emph{generic $d$-dimensional rigidity matroid} $\mathcal{R}_d(G)$ of $G$ is the row matroid of $R(G,p)$ for a generic embedding $p : V(G) \to \RR^d$. It is known that this is well-defined, and that in this way we obtain a nontrivial graph matroid family $\mathcal{R}_d$ with dimensionality $d$. We call $\mathcal{R}_d$ the \emph{family of generic $d$-dimensional rigidity matroids}.

The combinatorial characterization of $\mathcal{R}_d$-rigid graphs is a difficult problem that is wide open already in the $d=3$ case. 
To explore this problem and to capture some key combinatorial properties of $\mathcal{R}_d$, Graver~\cite{graver_1991} introduced the notion of an abstract rigidity matroid. Let $n$ and $d$ be positive integers with $n \geq d$, and let $\MM_0$ be a matroid on the edge set of $K_n$ with closure operator $\cl_0$. We say that $\MM_0$ is a \emph{$d$-dimensional abstract rigidity matroid} if it satisfies the following pair of conditions.

\begin{itemize}
    \item If $E_1,E_2 \subseteq E(K_n)$ and $|V(E_1) \cap V(E_2)| \leq d-1$, then $\cl_0(E_1 \cup E_2) = \cl_0(E_1) \cup \cl_0(E_2)$.
    \item If $E_1,E_2 \subseteq E(K_n)$ and $V_i = V(E_i), i \in \{1,2\}$ are such that $|V_1 \cap V_2| \geq d$ and $\cl_0(E_i) = E(K_{V_i}), i \in \{1,2\}$, then $\cl_0(E_1 \cup E_2) = E(K_{V_1 \cup V_2})$.
\end{itemize}

Nguyen gave the following characterization of abstract rigidity matroids.
\begin{theorem}\cite[Lemma 2.1 and Theorem 2.2]{nguyen_2010}\label{theorem:nguyen}
    Let $n$ and $d$ be positive integers with $n \geq d$, and let $\MM_0$ be a matroid on the edge set of $K_n$. 
    \begin{enumerate}
        \item $\MM_0$ is a $d$-dimensional abstract rigidity matroid if and only if its rank is $dn - \binom{d+1}{2}$ and the addition of vertices of degree at most $d$ preserves independent sets in $\MM_0$.
        \item If $\MM_0$ is an abstract rigidity matroid, then every copy of $E(K_{d+2})$ in $K_n$ is a circuit of $\MM_0$.
    \end{enumerate}
\end{theorem}

We stress that an abstract rigidity matroid need not be symmetric, i.e., invariant under automorphisms of $K_n$. Nonetheless, the most studied abstract rigidity matroids do have this property, and in fact, they are usually part of a graph matroid family. This motivates the investigation of families of abstract rigidity matroids.
 
We say that a graph matroid family $\MM$ is a \emph{family of $d$-dimensional abstract rigidity matroids} if $\MM(K_n)$ is a $d$-dimensional abstract rigidity matroid, for every $n \geq d$. It is immediate from \cref{theorem:nguyen}(a) that in this case $\MM$ is nontrivial and has dimensionality $d$. 

The following result provides a simple characterization of families of abstract rigidity matroids. We note that this result is implicit in~\cite[Section 10]{kalai_1985}. In fact, the notion of \emph{$d$-hypergraphic matroids} introduced in~\cite{kalai_1985} coincides with the notion of $d$-dimensional abstract rigidity matroid families.

\begin{proposition}\label{theorem:abstractrigidity}
    Let $\MM$ be an unbounded, nontrivial graph matroid family with dimensionality $\dimensionality{}$ and threshold $\threshold{}$. The following are equivalent:
    \begin{enumerate}
        \item $\MM$ is a family of abstract $\dimensionality{}$-dimensional rigidity matroids,
        \item $K_{d+2}$ is an $\MM$-circuit,
        \item we have $\threshold{} = \dimensionality{} + 1$.
    \end{enumerate}
\end{proposition}
\begin{proof}
    Let $r$ denote the rank function of $\MM$. The first implication is given by \cref{theorem:nguyen}(b), while \textit{(b)} $\Rightarrow$ \textit{(c)} follows from the definition of the threshold. Finally, \textit{(c)} $\Rightarrow$ \textit{(a)} follows from the observation that by \cref{lemma:0extension}, $K_{d+1}$ is $\MM$-independent, and hence \cref{lemma:ranklinear} implies that \[r(K_n) = d(n-d - 1) + \binom{d+1}{2} = dn - \left(d(d+1) - \binom{d+1}{2}\right) = dn - \binom{d+1}{2},\] for every $n \geq d$. Thus $\MM$ is an abstract rigidity matroid family by \cref{theorem:nguyen}(a).
\end{proof}

We can also show that families of abstract $\dimensionality{}$-dimensional rigidity matroids are the graph matroid families of dimensionality $\dimensionality{}$ that have, in a sense, smallest rank.

\begin{lemma}\label{lemma:minimumrank}
    Let $\MM$ be an unbounded, nontrivial graph matroid family with dimensionality $\dimensionality{}$.
    \begin{enumerate}
        \item For every $n \geq d+2$, we have $r(K_n) \geq dn - \binom{d+1}{2}$, and equality holds for some $n$ if and only if $\MM$ is a family of abstract $\dimensionality{}$-dimensional rigidity matroids.
        \item (See~\cite[Theorem 3.11.8]{graver.etal_1993}) The only $1$-dimensional family of abstract rigidity matroids is the family of graphic matroids.
    \end{enumerate}
\end{lemma}
\begin{proof}
    \textit{(a)} By the definition of dimensionality, $K_{d+1}$ is $\MM$-independent, and the addition of vertices of degree $d$ preserves $\MM$-independence. It follows that for every $n \geq d+2$, we have \[r(K_n) \geq d(n - d -1) + \binom{d+1}{2} = dn - \binom{d+1}{2},\] as claimed. Moreover, if $\MM$ is not a family of abstract $\dimensionality{}$-dimensional rigidity matroids, then by \cref{theorem:abstractrigidity}, $K_{\dimensionality{}+2}$ is $\MM$-independent, and hence by an analogous reasoning we have \[r(K_n) \geq d(n - d -2) + \binom{d+2}{2} > dn - \binom{d+1}{2}\] for every $n \geq d+2$.

    \textit{(b)} We will show that if $\MM$ is a $1$-dimensional abstract rigidity matroid, then the $\MM$-circuits are precisely the cycle graphs $C_n, n \geq 3$. In fact, it suffices to show that cycle graphs are $\MM$-circuits. Indeed, $\MM$ has dimensionality one, so forests are $\MM$-independent, and since every graph that is not a forest contains a cycle, if cycles are $\MM$-circuits, then they must be the only $\MM$-circuits.

    We prove that the cycle $C_n$ is an $\MM$-circuit by induction on $n$. The $n=3$ case follows from \cref{theorem:abstractrigidity}. Thus, let us assume that $n \geq 4$ and that $C_{n-1}$ is an $\MM$-circuit. By gluing a triangle onto an edge $e$ of $C_{n-1}$, we can write $C_n$ as $(C_{n-1} \cup C_3) - e$. Thus by the circuit elimination axiom, $C_n$ must contain an $\MM$-circuit. Since forests are $\MM$-independent, this $\MM$-circuit must be all of $C_n$, as desired.  
\end{proof}

Besides the family $\mathcal{R}_d$ of $d$-dimensional generic rigidity matroids, the most studied families of abstract rigidity matroids are the family $\mathcal{H}_d$ of $d$-hypergraphic matroids introduced by Kalai~\cite{kalai_1985}, and the family $\mathcal{C}_d$ of generic $C^{d-2}_{d-1}$-cofactor matroids introduced by Whiteley~\cite{whiteley_1996}. These are also defined as row matroids of certain symbolic matrices. 

Another way of constructing families of abstract rigidity matroids is by fixing an irreducible affine variety $X \subseteq \RR^d$ and letting $\mathcal{R}_X(G)$ be the row matroid of the rigidity matrix $R(G,p)$, where $p : V(G) \to X$ is a generic embedding. Here, ``generic'' is meant in the sense of ``generic among embeddings of $G$ into $X$'', as opposed to a generic embedding of $G$ into $\RR^d$ in which the points happen to lie on $X$. It can be shown using standard techniques that if $X$ is not contained in any affine hyperplane, then $\mathcal{R}_X$ is a family of abstract $d$-dimensional rigidity matroids. 

By taking $X = \RR^d$, we recover the family $\mathcal{R}_d$ of generic rigidity matroids. 
As far as we are aware, the only other case of this construction that has been considered previously is when $X$ is the moment curve in $\RR^d$, which was recently investigated by Crespo Ruiz and Santos~\cite{cresporuiz.santos_2023}.

\paragraph*{Bounded matroid families.}

The matroid families we have seen so far all have positive dimensionality. Let us also consider some bounded families. The most basic example is the family $\mathcal{U}_k$ of rank-$k$ uniform matroids, defined by letting $\mathcal{U}_k(G)$ be the rank-$k$ uniform matroid on $E(G)$, for every graph $G$. More generally, for any graph matroid family $\MM,$ we can take its \emph{rank-$k$ truncation} $\MM_k$, so that a graph $G$ is $\MM_{k}$-independent if and only if $G$ is $\MM$-independent and $|E(G)| \leq k$.

However, not all bounded graph matroids arise as the truncation of some unbounded graph matroid family. Indeed, if $\MM$ is unbounded, then every $\MM$-circuit has minimum degree at least two. Since every small $\MM_k$-circuit is also an $\MM$-circuit, this means that small $\MM_k$-circuits must also have minimum degree at least two. (Recall that an $\MM_k$-circuit is \emph{small} if it has at most $r(\MM_k)$ edges.) This is not true for all bounded graph matroid families, as shown by the next construction.

Let $k$ be a positive integer, and let $\mathcal{X}$ be a set of (isomorphism classes of) graphs, each with $k$ edges. We say that $\mathcal{X}$ is \emph{union-stable} if for any pair of graphs $G,H \in \mathcal{X}$ and any $e \in E(G) \cap E(H)$, either $(G \cup H) - e$ has at least $k+1$ edges, or $(G \cup H) - e \in \mathcal{X}$. It is not difficult to show that for such an $\mathcal{X}$, the assignment 
\[G \text{ is } \mathcal{U}_\mathcal{X}\text{-independent}
\Leftrightarrow |E(G)| \leq k \text{ and } G \notin \mathcal{X}\]
defines a bounded graph matroid family $\mathcal{U}_\mathcal{X}$, see~\cite[Lemma 2.5]{jackson.tanigawa_2024}.

For example, let $\mathcal{X}$ consist of (isomorphism classes of) the star $K_{1,k}$. If $k \geq 3$, then $\mathcal{X}$ is union stable, and in the associated matroid family $\mathcal{U}_\mathcal{X}$, $K_{1,k}$ is the only small $\MM$-circuit.

Another, perhaps more interesting example of a bounded graph matroid family can be constructed as follows.
Let $t$ be a positive integer, let $G$ be a graph, and fix a generic embedding $p: V(G) \to \RR^t$. The set of symmetric tensor products $\{p(u) \otimes p(v) + p(v) \otimes p(u), uv \in E(G)\}$, viewed as vectors in $\RR^t \otimes \RR^t$, defines a linear matroid $\mathcal{S}_t(G)$ on the edge set of $G$. This yields a graph matroid family $\mathcal{S}_t$ which we call the family of \emph{rank-$t$ generic symmetric tensor matroids}. It was recently shown by Brakensiek et al.\ \cite{brakensiek.etal_2024} that $\mathcal{S}_t(K_n)$ is dual to the generic rigidity matroid $\mathcal{R}_{n-t-1}(K_n)$. From this it follows (and it is not difficult to see directly) that $\mathcal{S}_t$ is a bounded family with rank $\binom{t+1}{2}$, and that $K_{1,t+1}$ is a small $\mathcal{S}_t$-circuit.

Jordán~\cite{jordan_2021} and Grasegger, Guler and Jackson~\cite{grasegger.etal_2022} characterized $\mathcal{R}_{d}(K_{d+t+1})$ for all $d$ whenever $t \leq 5$. By duality, this yields a characterization of $\mathcal{S}_t(K_n)$ for all $n$, and hence of $\mathcal{S}_t(G)$ for all graphs $G$, for $t \leq 5$. (See~\cite{jackson.tanigawa_2025} for an alternative proof of this result.) No combinatorial characterization is known when $t \geq 6$.

\paragraph*{Unions and de-coning.}
To end this subsection, we briefly consider a pair of operations for constructing new graph matroid families from old ones. 
Let $\mathcal{M}_0^i = (E,\mathcal{I}_i), i \in \{1,\ldots,k\}$ be a collection of matroids on a common ground set $E$. The \emph{union} of $\mathcal{M}_0^1,\ldots,\mathcal{M}_0^k$, denoted by $\cup_{i=1}^k \MM^i_0$, is the matroid on ground set $E$ whose family of independent sets is
\[\mathcal{I} = \{I_1 \cup \ldots \cup I_k : I_1 \in \mathcal{I}_1, \ldots, I_k \in \mathcal{I}_k\}.\]

Given graph matroid families $\MM_1,\ldots,\MM_k$, we define their \emph{union} $\MM = \cup_{i=1}^k \MM_i$ by letting $\MM(G) = \cup_{i=1}^k \MM_i(G)$ for every finite graph $G$. It is immediate from the definitions that $\MM$ is also a graph matroid family. The following lemma shows that dimensionality is additive under taking unions. %

\begin{lemma}\label{lemma:uniondimensionality}
    Let $\MM$ be the union of the graph matroid families $\MM_1,\ldots,\MM_k$. If each of $\MM_1,\ldots,\MM_k$ is nontrivial, then so is $\MM$. In this case the dimensionality of $\MM$ is the sum of the dimensionalities of $\MM_1,\ldots,\MM_k$. 
\end{lemma}
\begin{proof}
    Suppose that each $\MM_i$ is nontrivial, let $d_i,t_i$ and $r_i$ denote its dimensionality, threshold, and rank function, respectively, and set $d = \sum_{i=1}^k d_i$ and $t = \max_{i=1}^k t_i$. Let $r$ be the rank function of $\MM$. For every integer $n \geq t$ we have \[r(K_n) \leq \sum_{i=1}^k r_i(K_n) = \sum_{i=1}^k d_i(n-t) + r_i(K_t) = d(n-t) + \sum_{i=1}^k r_i(K_t),\]where the first equality uses \cref{lemma:ranklinear}. Thus $\MM$ is nontrivial, and it follows from the same lemma that its dimensionality is at most $d$. 

    To show that the dimensionality of $\MM$ is at least $d$, it suffices by \cref{lemma:0extension} to show that addition of vertices of degree $d$ preserves $\MM$-independence. This follows immediately from the facts that addition of vertices of degree $d_i$ preserves $\MM_i$-independence, and that a graph is $\MM$-independent if and only if it can be written as the union of $\MM_i$-independent graphs.
\end{proof}
In particular, \cref{lemma:uniondimensionality} implies that if some of $\MM_1,\ldots,\MM_k$ is unbounded, then so is $\MM$.

For the threshold of the union, we have the following bound. The proof is an adaptation of~\cite[Lemma 2.1]{garamvolgyi.etal_2024a}.

\begin{lemma}\label{lemma:unionthresholdbound}
    Let $\MM$ be the union of the graph matroid families $\MM_1,\ldots,\MM_k$. Suppose that for each $i \in \{1,\ldots,k\}$, $\MM_i$ is nontrivial and has dimensionality ${\dimensionality{}}_i$ and threshold ${\threshold{}}_i$. Let $\threshold{} = \sum_{i=1}^k \max({\threshold{}}_i,2{\dimensionality{}}_i)$.
    \begin{enumerate}
        \item If a graph $G$ contains edge-disjoint $\MM_i$-rigid spanning subgraphs $G_i, i \in \{1,\ldots,k\}$, then $G$ is $\MM$-rigid.
        \item Conversely, if $G$ is $\MM$-rigid and has at least $\threshold{}$ vertices, then $G$ contains edge-disjoint $\MM_i$-rigid spanning subgraphs $G_i, i \in \{1,\ldots,k\}$.
        \item The threshold of $\MM$ is at most $\threshold{}$.
    \end{enumerate}
\end{lemma}
\begin{proof}
    Part \textit{(a)} follows from the observation that
    \[r(G) \leq r(K_{V(G)}) \leq \sum_{i=1}^k r_i(K_{V(G)}) = \sum_{i=1}^k r_i(G_i) \leq r(G),\] where $r$ and $r_1,\ldots,r_k$ denote the rank function of $\MM$ and $\MM_1,\ldots,\MM_k$, respectively.

    For \textit{(b)}, let $n$ denote the number of vertices of $G$. It suffices to consider the case $G = K_n$. Indeed, if $K_n$ contains edge-disjoint $\MM_i$-rigid spanning subgraphs, then \[r(G) = r(K_n) = \sum_{i=1}^k r_i(K_n),\] and hence $G$ must also contain suitable spanning subgraphs.
    We construct the subgraphs $G_1,\ldots,G_k$ as follows. Let us partition the vertex set of $K_n$ into sets of vertices $V_i^j, i \in \{1,\ldots,k\}, j \in \{1,2\}$ such that for all $i \in \{1,\ldots,k\}$, $|V_i^1| \geq {\dimensionality{}}_i,|V_i^2| \geq {\dimensionality{}}_i,$ and $|V_i^1 \cup V_i^2| \geq {\threshold{}}_i$. For each $i \in \{1,\ldots,k\}$, let $G_i$ consist of the complete graph on $V_i^1 \cup V_i^2$, plus the edge sets \[\bigcup_{\ell<i}\left(E(V^1_i,V^1_\ell) \cup E(V^2_i,V^2_\ell)\right) \cup \bigcup_{\ell > i}\left(E(V^1_i,V^2_\ell) \cup E(V^2_i,V^1_\ell) \right).\] Now by \cref{lemma:ranklinear}, $G_i$ is $\MM_i$-rigid, since it has a spanning subgraph obtained from a complete graph on at least ${\threshold{}}_i$ vertices by adding vertices of degree at least ${\dimensionality{}}_i$. It is also easy to see that $G_i$ and $G_j$ are edge-disjoint for $i \neq j$.

    For \textit{(c)}, fix edge-disjoint $\MM_i$-rigid spanning subgraphs $G_i, i \in \{1,\ldots,k\}$ in $K_{\threshold{}}$, and let $G$ denote the union of $G_1,\ldots,G_k$. Since $\threshold{} \geq {\threshold{}}_i$, adding a new vertex of degree $d_i$ to $G_i$ also results in an $\MM_i$-rigid graph, and thus adding a new vertex $v$ of degree $d$ to $G$ also results in a graph $G'$ that has edge-disjoint $\MM_i$-rigid spanning subgraphs for $i \in \{1,\ldots,k\}$. By \textit{(a)}, this implies that $G'$ is $\MM$-rigid. Hence adding another edge between $v$ and some vertex of $G$ must create an $\MM$-circuit $C$ in $G'$ such that $1 \leq d_C(v) \leq d+1$. By the definitions of the dimensionality and the threshold, we must have $d_C(v) = d+1$ and \[\threshold{\MM} \leq |V(C)| - 1 \leq |V(G')| - 1 = \threshold{}.\vspace{-1em}\] 
\end{proof}

\cref{lemma:unionthresholdbound}(c) gives a tight bound on the threshold for the $k$-fold union of the family $\graphicmatroid$ of graphic matroids. However, the bound is not tight in general, even for count matroids. See~\cite[Lemma 5.4]{garamvolgyi.etal_2024a} for a better bound for the union $\mathcal{R}_d \cup \graphicmatroid$.

There is no clear way to define the dual of a graph matroid family $\MM$. Indeed, taking the dual matroid $\MM(G)^*$ for each graph $G$ gives a well-defined, but not compatible family of matroids. That is, the restriction of $\MM(G)^*$ to the edge set of a subgraph $H$ of $G$ does not, in general, coincide with $\MM(H)^*$; instead, it is dual to the matroid obtained from $\MM(G)$ by contracting (in a matroidal sense) the edges not in $H$. 

Similarly, we cannot, in general, define the contraction of a graph matroid family. However, given an unbounded graph matroid family $\MM$, we can define another graph matroid family $\MM'$ by letting $\MM'(K_n)$ be the matroid obtained by the contraction of a vertex star in $\MM(K_{n+1})$. More explicitly, a graph is $\MM'$-independent if and only if its cone graph is $\MM$-independent, where the \emph{cone graph} of a graph $G$ is obtained by adding a new vertex to $G$ that is connected to every vertex of $G$. 

Let us call the construction of $\MM'$ from $\MM$ the \emph{de-coning} operation. 
It is a classical result of Whiteley~\cite{whiteley_1983} that by applying this operation to the family of $(d+1)$-dimensional generic rigidity matroids, we recover the family of $d$-dimensional generic rigidity matroids. It is also not difficult to see that by de-coning the family of $(k,\ell)$-count matroids, where $k \geq 2$, we obtain the family of $(k-1,\ell-k)$-count matroids. Since we will not investigate this operation further, we omit the details.

\section{The Lovász-Yemini and Whitney properties}\label{section:main}

We are now ready to begin our investigation of the main concepts of the paper: the Lovász-Yemini and the Whitney properties. We start by giving a rigorous definition of reconstructibility with respect to a graph matroid family. 

Let $\MM$ be a graph matroid family, and let $G$ be a graph. We say that $G$ is \emph{$\MM$-reconstructible} if $\mathcal{M}(G)$ uniquely determines $G$, in the sense that if $H$ is another graph and $\psi: E(G) \rightarrow E(H)$ is an isomorphism between $\MM(G)$ and $\MM(H)$, then $\psi$ is induced by a graph isomorphism. (A graph isomorphism $\varphi: V(G) \to V(H)$ \emph{induces} $\psi$ if $\psi(uv) = \varphi(u)\varphi(v)$ holds for every edge $uv \in E(G)$.) 

We say that a graph matroid family $\MM$ has the \emph{Whitney property (with constant $c$)} if there is a nonnegative integer $c$ such that every $c$-connected graph is $\MM$-reconstructible.
Similarly, let us say that a graph matroid family $\MM$ has the \emph{Lovász-Yemini property (with constant $c$)} if there is a nonnegative integer $c$ such that every $c$-connected graph is $\MM$-rigid.

In the next subsection we prove \cref{theorem:main}, which says that for an unbounded graph matroid family, the Whitney property is equivalent to the Lovász-Yemini property. The bounded case is investigated in \cref{subsection:boundedcase}: we observe that every bounded graph matroid family has the Lovász-Yemini property, and we give a characterization of the Whitney property for a bounded graph matroid family $\MM$ in terms of the existence of certain $\MM$-circuits. In \cref{subsection:unionreconstructibility}, we show that taking unions of graph matroid families preserves both the Whitney and the Lovász-Yemini properties. Finally, in \cref{subsection:extendable}, we define the class of $1$-extendable graph matroid families and show that every $1$-extendable graph matroid family has the Lovász-Yemini (and thus the Whitney) property.

\subsection{The unbounded case}

We begin by proving the easy direction of \cref{theorem:main}; namely, that for an unbounded graph matroid family, the Whitney property implies the Lovász-Yemini property. The key observation is the following.

\begin{lemma}\label{lemma:reconstructiblebridge}
    If $\MM$ is an unbounded graph matroid family and $G$ is an $\MM$-reconstructible graph with at least two edges, then $\MM(G)$ is bridgeless.
\end{lemma}
\begin{proof}
    Suppose in the contrapositive that $\MM(G)$ contains a bridge $e$. Let $v \in V(G)$ be a vertex not incident to $e$, and let $u_1,u_2$ be a pair of new vertices. It follows from \cref{lemma:0extension} that $vu_1$ is a bridge in $G_1 = G-e+vu_1$, and similarly, $u_1u_2$ is a bridge in $G_2 = G-e+u_1u_2$. Hence $\MM(G)$, $\MM(G_1)$ and $\MM(G_2)$ are all isomorphic. Since $G$ cannot be isomorphic to both $G_1$ and $G_2$, we conclude that $G$ is not $\MM$-reconstructible.     
\end{proof}

\begin{lemma}\label{lemma:whitneytoLY}
    If an unbounded graph matroid family $\MM$ has the Whitney property with constant $c$, then it has the Lovász-Yemini property with constant $c$.
\end{lemma}
\begin{proof}
    Let us suppose, for a contradiction, that every $c$-connected graph is $\MM$-reconstructible, but there is some $c$-connected graph that is not $\MM$-rigid. Then there is a pair of nonadjacent vertices $u,v$ in $G$ such that $uv$ is a bridge in $\MM(G+uv)$. Now $G+uv$ is a $c$-connected, and hence $\MM$-reconstructible graph on at least two edges, such that $\MM(G+uv)$ contains a bridge. But this contradicts \cref{lemma:reconstructiblebridge}. 
\end{proof}

Our main goal in this subsection is to prove the other direction of \cref{theorem:main}, in the following, more precise form. 

\begin{theorem}\label{theorem:mainprecise}
    Let $\MM$ be a nontrivial, unbounded graph matroid family with threshold $t$ and rank function $r$. If $\MM$ has the Lovász-Yemini property with constant $c$, then $\MM$ has the Whitney property with constant $c' = \max(c,t) + r(K_{\max(c,t)-1}) + 2$.
\end{theorem}

Our proof relies on the notion of vertical connectivity of matroids, which we introduce next.

Let $\mathcal{M}_0 = (E,r)$ be a matroid with rank function $r$, and let $k$ be a positive integer.
We say that a bipartition $(E_1,E_2)$ of $E$ is a \emph{vertical $k$-separation} of $\mathcal{M}_0$ if 
\begin{equation}\label{eq:vertseparator1}
    r(E_1), r(E_2) \geq k    
\end{equation}
and 
\begin{equation}\label{eq:vertseparator2}
    r(E_1) + r(E_2) \leq r(E) + k -1 
\end{equation}
hold. Note that \cref{eq:vertseparator1} and \cref{eq:vertseparator2} together imply $r(E_i) < r(E)$ for $i \in \{1,2\}$.
We say that $\mathcal{M}_0$ is \emph{vertically $k$-connected} if $k \leq r(\MM_0)$ and $\mathcal{M}_0$ does not have vertical $k'$-separations for any positive integer $k' < k$. For $k=1$, the latter condition is vacuous, and hence every matroid with positive rank is vertically $1$-connected.

The following lemma shows that to certify that a matroid is not vertically $k$-connected, it suffices to consider ``vertical separations'' $E = E_1 \cup E_2$ in which $E_1$ and $E_2$ are not necessarily disjoint. 

\begin{lemma}\label{lemma:verticalconnsimpler}
    Let $\mathcal{M}_0 = (E,r)$ be a matroid, and let $E_1,E_2 \subseteq E$ be subsets of $E$, not necessarily disjoint, with $E = E_1 \cup E_2$. Suppose that $r(E_1), r(E_2) \geq k$ and $r(E_1) + r(E_2) \leq r(E) + k - 1$. Then $\mathcal{M}_0$ has a vertical $k'$-separation for some $k' \leq k$.
\end{lemma}
\begin{proof}
    Let $E'_2 = E_2 - E_1$. Then $(E_1,E'_2)$ is a bipartition of $E$, and it is straightforward to check that it is a vertical $k'$-separation for $k' = k - (r(E_2) - r(E_2'))$.
\end{proof}

It is well-known that a graph without isolated vertices is $k$-connected if and only if its graphic matroid is vertically $k$-connected (see, e.g.,~\cite[Theorem 8.6.1]{oxley_2011}).
The main ingredients to our proof of \cref{theorem:mainprecise}, \cref{theorem:vertconntoconn} and \cref{theorem:vertexredundantrigidtovertconn} below, are partial generalizations of this fact to arbitrary graph matroid families. \cref{theorem:vertconntoconn} states that given an unbounded, nontrivial graph matroid family $\MM$, if $\MM(G)$ has sufficiently high vertical connectivity, then $G$ must have high vertex-connectivity (provided that it has no isolated vertices).
Conversely, \cref{theorem:vertexredundantrigidtovertconn} shows that if $G$ has high vertex-connectivity and is highly vertex-redundantly $\MM$-rigid, then $\MM(G)$ must have high vertical connectivity. Finally, if $\MM$ has the Lovász-Yemini property, then high vertex-connectivity implies high vertex-redundant $\MM$-rigidity. By combining this triangle of implications with ideas from previous Whitney-type results, we can deduce that the Lovász-Yemini property implies the Whitney property.

Let us note that this proof strategy roughly follows the strategy used in~\cite{garamvolgyi.etal_2024} to show that the family of generic $C^1_
2$-cofactor matroids has the Whitney property. The crucial difference is that in that case, a combinatorial characterization of the rank function was available, while here we can only use the assumptions that $\MM$ is unbounded, nontrivial, and has the Lovász-Yemini property.

We start by giving a bound on the number of vertices of a graph $G$, provided that $\MM(G)$ is vertically $k$-connected. The proof of this simple statement turns out to be surprisingly subtle.

\begin{lemma}\label{lemma:verticalconnbound}
    Let $G = (V,E)$ be a graph, and let $\MM$ be a nontrivial, unbounded graph matroid family with dimensionality $\dimensionality{}$.  
    Let $k \geq 2$ be an integer. If $\MM(G)$ is vertically $k$-connected, then its minimum degree is at least $k+ \dimensionality{} -1$. In particular, $|V| \geq k + \dimensionality{}$.
\end{lemma}
\begin{proof}
    Let $r$ denote the rank function of $\MM$. 
    Note that since $\MM(G)$ is vertically $2$-connected, it is bridgeless, and hence it follows from \cref{lemma:0extension} that every vertex of $G$ has degree at least $\dimensionality{}+1$. 
    
    Let $v \in V$ be a vertex of smallest degree in $G$. Set $k_0 = d_G(v) - \dimensionality{} + 1$, and let $E_1$ be a set of $k_0$ edges incident to $v$. (Note that $\dimensionality{} \geq 1$ and thus $k_0 \leq d_G(v)$.)
    We will show that $(E_1,E - E_1)$ is a vertical $k_0$-separation of $\MM(G)$. Since the latter is vertically $k$-connected, it follows that $k_0 \geq k$, and hence that $d_G(v) \geq k + \dimensionality{} - 1$, as claimed.

    To see that $(E_1, E- E_1)$ is a vertical $k_0$-separation, note that by \cref{lemma:unboundedforest}, $E_1$ is independent in $\MM(G)$, and hence $r(E_1) = k_0$. Moreover, since $\deg_{E - E_0}(v) \leq d-1$, it follows from \cref{lemma:0extension} that adding any edge to $E - E_1$ incident to $v$ increases the rank, and thus we have $r(E) \geq r(E - E_1) + 1$ and \[r(E_1) + r(E - E_1) \leq r(E) + k_0 - 1.\] Hence all that is left to show is that $r(E - E_1) \geq k_0$.

    First, suppose that either $\dimensionality{} \geq 2$ or $v$ is not adjacent to every other vertex in $G$. In both cases we can find a vertex $w$ such that $d_{E-E_1}(w) = d_G(w)$, and thus, using \cref{lemma:unboundedforest} again, \[r(E - E_1) \geq d_{E - E_1}(w) = d_G(w) \geq d_G(v) \geq k_0,\] as claimed.  

    Thus we may assume that $\dimensionality{} = 1$ and that $v$ is adjacent to every vertex in $G-v$. Since $v$ has smallest degree, it follows that $G$ is a complete graph. Note that in this case $k_0 = d_G(v) = |V(G-v)|$. Now if $\mathcal{M}$ is not the family of graphic matroids, then we can use \cref{lemma:minimumrank} to deduce that $r(E-E_1) = r(G-v) \geq |V(G-v)| = k_0$, as required. Finally, if $\MM$ is the family of graphic matroids, then by the definition of vertical $k$-connectivity, we have $k + 1 \leq r(G) + 1 = |V|$, and hence the minimum degree of $G$ is $|V| - 1 = k$, as desired.
\end{proof}

\begin{proposition}\label{theorem:vertconntoconn}
    Let $G = (V,E)$ be a graph, and let $\MM$ be a nontrivial, unbounded graph matroid family with dimensionality $\dimensionality{}$ and threshold $\threshold{}$. 
    Let $k$ be an integer with $k \geq \threshold{}$. If $\MM(G)$ is vertically $(r(K_k) + 2)$-connected, then $G$ is $(k+1)$-connected.
\end{proposition}
\begin{proof}
    It follows from \cref{lemma:minimumrank} that $r(K_k) \geq k - 1$, and hence from \cref{lemma:verticalconnbound} that $|V| \geq r(K_k) + \dimensionality{} + 2 \geq k + 2$. 
    Suppose, for a contradiction, that there exists a set $S \subseteq V$ of vertices of size at most $k$ such that $G-S$ is disconnected.
    By adding vertices to $S$, we may suppose that $|S| = k$. (Here we use the fact that $|V| \geq k+2$.) 
    
    Let $V_1$ be the union of some, but not all components of $G - S$, and let $V_2 = V - (V_1 \cup S)$. 
    Set $G_i = G[V_i \cup S]$ and $G_i' = G_i + E(K_S)$, for $i \in \{1,2\}$. Finally, let 
    \[ a_i = r(K_{V_i \cup S}) - r(G_i), \hspace{1em} \text{and} \hspace{1em} b_i = r(K_{V_i \cup S}) - r(G'_i),\]for $i \in \{1,2\}$. 
    By symmetry, we may assume that $a_1 - b_1 \leq a_2 - b_2$.
    Note that since $k \geq \threshold{}$, we have, by \cref{lemma:ranklinear},
    \begin{align*}
        r(K_{V_i \cup S}) &= d|V_i| + r(K_k), \hspace{1em} i \in \{1,2\}, \text{ and } \\
        r(K_V) &= \dimensionality{}|V_1 \cup V_2| + r(K_k).
    \end{align*}

    We shall show that
    \begin{equation}\label{eq:first}
        r(G_i) \geq r(K_k) - (a_2 - b_2) + 1, \hspace{1em} i \in \{1,2\},
    \end{equation}
    and 
    \begin{equation}\label{eq:second}
        r(G_1) + r(G_2) \leq r(G) + r(K_k) - (a_2 - b_2).
    \end{equation}
    \cref{lemma:verticalconnsimpler}, \cref{eq:first}, and \cref{eq:second} together imply that $\MM(G)$ has a vertical $c$-separation for some positive integer $c$ with \[c \leq r(K_k) - (a_2 - b_2) + 1 \leq r(K_k) + 1,\] contradicting our assumption that $\MM(G)$ is vertically $(r(K_k) + 2)$-connected.

    To show that \cref{eq:first} holds, note that $r(G_i) = d|V_i| + r(K_k) - a_i$. Hence (using the assumption that $a_1 - b_1 \leq a_2 - b_2$) it suffices to show that
    \[ d|V_i| + r(K_k) - a_i \geq r(K_k) - (a_i - b_i) + 1,\] or equivalently, that $d|V_i| - 1 \geq b_i$, for $i \in \{1,2\}$. Since $S$ is a clique in $G_i'$, and $G_i'$ contains at least one edge not spanned by $S$, we have $r(G_i') \geq r(K_k) + 1$. It follows that 
    \[b_i = r(K_{V_i \cup S}) - r(G_i') \leq d|V_i| + r(K_k) - (r(K_k) + 1) = d|V_i| - 1,\]as required.

    It only remains to show that \cref{eq:second} holds. We have
    \begin{align*}
    r(G_1) + r(G_2) &= d|V_1| + r(K_k) - a_1 + d|V_2| + r(K_k) - a_2 \\ &= d|V_1 \cup V_2| + 2r(K_k) - (a_1 + a_2) \\
    &= r(K_V) + r(K_k) - (a_1+a_2)
    \\ &= r(K_V) - (a_1 + b_2) + r(K_k) - (a_2 - b_2).
    \end{align*}
    Hence we only need to show that
    \[r(K_V) - (a_1 + b_2) \leq r(G),\]or equivalently, that we can make $G$ $\MM$-rigid by adding at most $a_1 + b_2$ suitable edges. 
    
    Let $A_1$ and $B_2$ be edge sets such that $|A_1| = a_1, |B_2| = b_2$ and $G_1 + A_1$ and $G_2' + B_2$ are $\MM$-rigid. 
    We claim that $G + A_1 + B_2$ is $\MM$-rigid. Indeed, since $G+A_1$ is $\MM$-rigid, connecting each pair of vertices in $V_1 \cup S$ by an edge does not increase the rank of $G+A_1+B_2$. Since the graph obtained in this way contains $G_2'+B_2$ as a subgraph, and the latter is $\MM$-rigid, we can further add the edges spanned by  $K_{V_2 \cup S}$ without increasing the rank. The resulting graph is the union of the two complete graphs $K_{V_1 \cup S}$ and $K_{V_2 \cup S}$, which intersect in $k \geq \threshold{}$ vertices. It follows from \cref{lemma:gluing}(a) that this graph is $\MM$-rigid, and hence so is $G + A_1 + B_2$, as required. 
\end{proof}

Given a graph matroid family $\MM$ and a nonnegative integer $k$, we say that a graph $G$ is \emph{$k$-vertex-redundantly $\MM$-rigid} if it is $\MM$-rigid and remains so after the deletion of any set of at most $k$ vertices. Note that the $0$-vertex-redundantly $\MM$-rigid graphs are precisely the $\MM$-rigid graphs.

\begin{proposition}\label{theorem:vertexredundantrigidtovertconn}
    Let $G = (V,E)$ be a graph and let $\MM$ be a nontrivial, unbounded graph matroid family with rank function $r$, dimensionality $\dimensionality{}$, and threshold $\threshold{}$. If $G$ is $k$-connected, $k$-vertex-redundantly $\MM$-rigid, and has at least $k + \threshold{}$ vertices, then $\MM(G)$ is vertically $(k+1)$-connected.
\end{proposition}
\begin{proof}
    We prove by induction on $k$; the $k = 0$ case is trivial. Let us thus assume that $k \geq 1$. Note that since $G$ is $\MM$-rigid, we have \[r(G) \geq r(K_{k + \threshold{}}) \geq k + t - 1 \geq k+1,\]
    where the second inequality follows from \cref{lemma:minimumrank}.

    Let us suppose, for a contradiction, that $\MM(G)$ has a vertical $k'$-separation $(E_1,E_2)$ for some $k' \leq k$. Since $E_1$ and $E_2$ are both nonempty and $G$ is connected, there exists a vertex $v \in V(E_1) \cap V(E_2)$. Note that $G-v$ is $(k-1)$-connected and $(k-1)$-vertex-redundantly $\MM$-rigid, and hence by the induction hypothesis, $\MM(G-v)$ is vertically $k$-connected. Also, $G-v$ and $G$ are both $\MM$-rigid, and since they both have at least $\threshold{}$ vertices, we have $r(G) = r(G-v) + d$. 
    
    Let $E_i' = E_i - \partial_G(v)$ for $i \in \{1,2\}$. From the assumption on $G$, we have \[d_G(v) \geq d + k \geq d + 1,\] for otherwise by deleting a suitable set of at most $k$ vertices, we would obtain an $\MM$-rigid graph on at least $\threshold{}$ vertices in which $v$ has degree less than $\dimensionality{}$, contradicting \cref{lemma:0extension}. We also have $d_{E_1}(v), d_{E_2}(v) \geq 1$. Using \cref{lemma:0extension}, we deduce that
    \begin{align*}
    r(E_1') + r(E_2') &\leq r(E_1) + r(E_2) - (d+1) \\ &\leq r(G) + (k' - 1) - d - 1  \\ &= r(G-v) + (k' - 1) - 1.
    \end{align*}
    It follows that
    \begin{equation}\label{eq:c}
    r(E_1') + r(E_2') = r(G-v) + c - 1
    \end{equation}
    for some integer $c$ with $1 \leq c \leq k'-1 \leq k - 1$.
    If $r(E_1'),r(E_2') \geq c,$ then $(E_1',E_2')$ is a vertical $c$-separation of $\MM(G-v)$, a contradiction. Thus we must have $r(E_1') \leq c - 1$ or $r(E_2') \leq c - 1$; by symmetry, we may suppose that it is the former. 
    
    Now \cref{eq:c} implies that $r(E_1') = c-1$ and $r(E_2') = r(G-v)$. 
    If $d_{E_2}(v) \geq d$, then by \cref{lemma:0extension}, we have $r(E_2) = r(E_2') + d = r(G-v) + d = r(G)$, contradicting the fact that $(E_1,E_2)$ is a vertical $k'$-separation. Hence $d_{E_2}(v) \leq d-1$. It follows from \cref{lemma:0extension,lemma:unboundedforest} that
    \begin{align*}
    r(E_1) + r(E_2) &\geq d_{E_1}(v) + r(E'_2) + d_{E_2}(v) = r(G-v) + d_G(v) \\ &\geq r(G-v) + d + k = r(G) + k \\ &> r(G) + k' - 1.
    \end{align*}
    But this contradicts, again, the fact that $(E_1,E_2)$ is a vertical $k'$-separation.
\end{proof}

Let us note that by combining \cref{theorem:vertconntoconn} and \cref{theorem:vertexredundantrigidtovertconn}, we recover the fact that a graph is $k$-vertex-connected if and only if its graphic matroid is vertically $k$-connected.

Note that if $\MM$ has the Lovász-Yemini property with constant $c$, then every $(c+k)$-connected graph is $k$-vertex-redundantly $\MM$-rigid.
Thus the following is an immediate corollary of \cref{theorem:vertexredundantrigidtovertconn}.
\begin{proposition}\label{theorem:conntovertconn}
    Let $G = (V,E)$ be a graph and let $\MM$ be a nontrivial, unbounded graph matroid family with rank function $r$, dimensionality $\dimensionality{}$, and threshold $\threshold{}$. Suppose that $\MM$ has the Lovász-Yemini property with constant $c$. If $G$ is $(\max(\threshold{},c) + k)$-connected for some positive integer $k$, then $\MM(G)$ is vertically $(k+1)$-connected.
\end{proposition}

We are now ready to prove \cref{theorem:mainprecise}. 
We shall use the following folklore statement.

\begin{lemma}\cite[Lemma 5.1]{garamvolgyi.etal_2024}\label{lemma:star}
    Let $G = (V,E)$ and $H = (V',E')$ be graphs without isolated vertices, and let $\psi : E \rightarrow E'$ be a bijection that ``sends stars to stars'', that is, such that for every $v \in V$, there is a vertex $v' \in V'$ with $\psi(\partial_G(v)) = \partial_H(v')$. Then $\psi$ is induced by a graph isomorphism.
\end{lemma}

\begin{proof}[Proof of \cref{theorem:mainprecise}]
    Let $d \geq 1$ denote the dimensionality of $\MM$.
    Let us fix a $c'$-connected graph $G$, let $H$ be a graph without isolated vertices, and let $\psi: E(G) \rightarrow E(H)$ be an isomorphism between $\MM(G)$ and $\MM(H)$. Our goal is to show that $\psi$ is induced by a graph isomorphism. 
    
    Fix $v \in V(G)$ and consider $F = E(G) - \partial_G(v)$. Set $F' = \psi(F)$ and $H' = H[F']$; note that $\MM(G-v)$ and $\MM(H')$ are isomorphic matroids. By \cref{theorem:conntovertconn}, $\MM(G)$ and $\MM(G-v)$ are both vertically $\left(r(K_{c-1}) + 2\right)$-connected, and hence \cref{theorem:vertconntoconn} implies that $H$ and $H'$ are both $c$-connected. It follows from the Lovász-Yemini property that $G, G-v, H$, and $H'$ are all $\MM$-rigid.

    Since $\MM$ is unbounded and $F$ is the edge set of an induced subgraph of $G$, it is closed in $\MM(G)$, and hence $F'$ is closed in $\MM(H)$. On the other hand, $F'$ induces the $\MM$-rigid graph $H'$, and thus adding edges of $H$ induced by $V(H')$ to $F'$ cannot increase its rank. We deduce that every such edge must already be contained in $F'$; in other words, that $H'$ is an induced subgraph of $H$.
    Also note that \[r(H) = r(G) = r(G-v)+d = r(H') + d.\]
    Since $H$ and $H'$ are both $\MM$-rigid and have at least $t$ vertices, it follows that $|V(H')| = |V(H)| - 1$.

    Thus $H'$, being an induced subgraph of $H$ on $|V(H)|-1$ vertices, can be written as $H - w$ for some $w \in V(H)$. This shows that $\psi$ maps complements of vertex stars to complements of vertex stars. Since $\psi$ is a bijection, it must also map vertex stars to vertex stars. \cref{lemma:star} now implies that $\psi$ is induced by a graph isomorphism, as desired. 
\end{proof}

\cref{theorem:main} follows immediately from \cref{lemma:whitneytoLY}, \cref{theorem:mainprecise}, and the observation that the trivial graph matroid family has neither the Lovász-Yemini nor the Whitney properties.

As an immediate consequence of \cref{theorem:vertconntoconn}, we can also deduce the following reconstructibility result, in which the requirement is not that $G$ is highly connected, but rather that $\MM(G)$ is highly vertically connected.

\begin{theorem}\label{theorem:verticalconnwhitney}
    Let $\MM$ be an unbounded graph matroid family with rank function $r$, and let $G$ be a graph. Suppose that $\MM$ has the Whitney property with constant $c$. If $\MM(G)$ is vertically $(r(K_{c-1})+2)$-connected, then $G$ is $\MM$-reconstructible.
\end{theorem}

Let us end this subsection by specializing the previous results to the family $\mathcal{R}_d$ of $d$-dimensional generic rigidity matroids. As noted in the introduction, Villányi recently established the Lovász-Yemini property for this family.

\begin{theorem}\label{theorem:Soma}\cite{villanyi_2023}
    Every $d(d+1)$-connected graph is $\mathcal{R}_d$-rigid.
\end{theorem}

By combining \cref{theorem:mainprecise,theorem:verticalconnwhitney,theorem:Soma}, we can deduce the following Whitney-type result for $\mathcal{R}_d(G)$. This answers positively a question of Brigitte and Herman Servatius~\cite[Problem 17]{alexandrov.etal_2010}.

\begin{theorem}\label{theorem:rigiditywhitney}
    Let $G$ be a graph. If $G$ is $d(d+1)^2$-connected, or $\mathcal{R}_d(G)$ is vertically $(d(d+1))^2$-connected, then $G$ is $\mathcal{R}_d$-reconstructible. 
\end{theorem}
\begin{proof}
    \cref{theorem:mainprecise,theorem:Soma} imply that $\mathcal{R}_d$ has the Whitney property with constant 
    \begin{align*}
        c' &= d(d+1) + d\big(d(d+1) - 1\big) - \binom{d+1}{2} + 2 \\ &= d(d+1)(d+1/2) - d + 2 \\ &\leq d(d+1)^2.
    \end{align*}
    This means that every $d(d+1)^2$-connected graph is $\mathcal{R}_d$-reconstructible. The second part of the statement follows from \cref{theorem:verticalconnwhitney} by noting that
    \begin{align*}
        d(c'-1) - \binom{d+1}{2} + 2 \leq d\big(d(d+1)^2-1\big) - \binom{d+1}{2} + 2 \leq \big(d(d+1)\big)^2.
    \end{align*}
\end{proof}

\cref{theorem:rigiditywhitney} shows that $\mathcal{R}_d$ has the Whitney property, for all $d$. Previously this was only known for $d=1$, given by Whitney's theorem, and for $d=2$, where Jordán and Kaszanitzky~\cite{jordan.kaszanitzky_2013} showed that the constant $7$ suffices.
It is apparent from the proof of \cref{theorem:rigiditywhitney} that the constant $d(d+1)^2$ is not tight, and we believe it to be far from optimal. In fact, for $d \geq 2$ it is open whether every $d(d+1)$-connected graph is $\mathcal{R}_d$-reconstructible. 

\subsection{The bounded case}\label{subsection:boundedcase}

Next, we consider rigidity and reconstructibility in bounded graph matroid families.
We start with a simple observation.

\begin{lemma}\label{lemma:boundedLY}
    Let $\MM$ be a graph matroid family. If $\MM$ is bounded, then it has the Lovász-Yemini property.
\end{lemma}
\begin{proof}
    Let $r$ denote the rank function of $\MM$, and let $k = 2(r(\MM) - 1) + 1$. It is not difficult to see that every $k$-connected graph contains a matching of size at least $r(\MM)$. It follows from \cref{lemma:unboundedforest}(b) that such a matching has rank $r(\MM)$. This implies that every $k$-connected graph has rank $r(\MM)$, and hence is $\MM$-rigid.
\end{proof}

On the other hand, not every bounded graph matroid family has the Whitney property. For example, the family $\mathcal{U}_k$ of rank-$k$ uniform matroids does not: in this case $\mathcal{U}_k(G)$ is determined by the number of edges of $G$, rather than $G$ itself. %

\cref{theorem:boundedwhitney} below provides a complete characterization of bounded graph matroid families with the Whitney property. It turns out that the example of $\mathcal{U}_k$ is characteristic, in the sense that if a bounded graph matroid family $\MM$ does not have the Whitney property, then there exist arbitrarily highly connected graphs $G$ for which $\MM(G)$ is a uniform matroid.

One of the main tools in the proof of \cref{theorem:boundedwhitney} is the following result from extremal graph theory.
For a finite family of graphs $\mathcal{C}$, let $ex(\mathcal{C},n)$ denote the maximum number of edges in a graph on $n$ vertices that does not contain a copy of any member of $\mathcal{C}$ as a subgraph. 

\begin{theorem}(See, e.g.,~\cite[Theorem 2.33]{Füredi2013})\label{theorem:turan}
    Let $\mathcal{C}$ be a finite family of graphs. If $\mathcal{C}$ does not contain a forest, then there exists a positive constant $c$ such that $ex(\mathcal{C},n) \geq n^{1+c}$.
\end{theorem}

\cref{theorem:turan} lets us construct $\mathcal{C}$-free graphs with arbitrarily high average degree. The following classical theorem of Mader allows us to replace ``high average degree'' with ``high vertex-connectivity''.

\begin{theorem}\label{theorem:mader}\cite{mader_1972} (see also~\cite[Theorem 1.4.3]{diestel_2017})
    Let $G$ be a graph, let $\delta = 2|E(G)|/|V(G)|$ denote the average degree of $G$, and let $k$ be a positive integer. If $\delta \geq 4k$, then $G$ contains a $k$-connected subgraph.
\end{theorem}

We shall also need the following simple combinatorial lemma.

\begin{lemma}\label{lemma:mindegone}
    Let $G$ be a graph and let $m \geq 2$ be an integer. If $|E(G)| \geq 2m - 2$, then $G$ contains a subgraph on $m$ edges and with minimum degree one.
\end{lemma}
\begin{proof}
    If $G$ has a vertex of degree at least $m$, then it contains the star $K_{1,m}$ as a subgraph. On the other hand, if $G$ has a vertex $v$ of degree at least $1$ and at most $m-1$, then $G-v$ has at least $m - 1$ edges. Hence we may take an arbitrary subgraph of $G-v$ on $m-1$ edges, plus an edge incident to $v$.
\end{proof}

Recall that given a bounded graph matroid family $\MM$ with rank $r(\MM)$, we call an $\MM$-circuit \emph{small} if it has at most $r(\MM)$ edges.

\begin{theorem}\label{theorem:boundedwhitney}
    Let $\MM$ be a bounded graph matroid family. Then $\MM$ has the Whitney property if and only if there is a small $\MM$-circuit with minimum degree one.
\end{theorem}
More precisely, we shall see that if $m$ is the minimum number of edges in a small $\MM$-circuit with minimum degree one, then $\MM$ has the Whitney property with constant $c = \max(2m-2,4)$. 
\begin{proof}
    First, suppose that there are no graphs among the small $\MM$-circuits with minimum degree one; in particular, this means that there are no small $\MM$-circuits that are forests. We shall show that for every positive integer $k$, there exists a $k$-connected graph $G$ that is not $\MM$-reconstructible. 
    
    Let $\mathcal{C}$ denote the set of (isomorphism classes of) small $\MM$-circuits. It follows from \cref{theorem:turan} that there exists a $\mathcal{C}$-free graph $G'$ with $2|E(G')|/|V(G')| \geq 4k$, and by \cref{theorem:mader}, $G'$ has a $k$-connected subgraph $G$. Since $G$ does not contain small circuits of $\MM$, every $\MM$-circuit in $G$ has size $r(\MM) + 1$. Combining this with the fact that every graph on $r(\MM) + 1$ edges is $\MM$-dependent, we deduce that $\MM(G)$ is a uniform matroid. It follows that any permutation of $E(G)$ induces an automorphism of $\MM(G)$. On the other hand, not every permutation of $E(G)$ is induced by an automorphism of $G$ (unless $G$ is a star or a triangle, which is ruled out for $k \geq 3$), since a permutation induced by an automorphism must map incident edges to incident edges. This means that $G$ is not $\MM$-reconstructible, as required.

    Now let us suppose that there is a small $\MM$-circuit with minimum degree one. By \cref{lemma:smallcircuits}, there is a positive integer $m \geq 2$ such that the star $K_{1,m}$ is the unique small $\MM$-circuit with minimum degree one on at most $m$ edges.

    Let $c = \max(2m-2,4)$. We shall show that every $c$-connected graph is $\MM$-reconstructible. In fact, we show that minimum degree at least $c$ suffices. Let $G$ be a graph with minimum degree at least $c$, and suppose that $\psi: E(G) \to E(H)$ is an isomorphism between $\MM(G)$ and $\MM(H)$, for some graph $H$ without isolated vertices. Fix a vertex $v \in V(G)$, and consider the subgraph $H'$ of $H$ induced by $\psi(\partial_G(v))$. Since $|E(H')| = d_G(v) \geq 2m - 2$, \cref{lemma:mindegone} implies that $H'$ has a subgraph $C$ on $m$ edges and with minimum degree one. 
    
    Note that since $K_{1,m}$ is an $\MM$-circuit, every subset of $m$ edges in $\partial_G(v)$ induces an $\MM$-circuit, and thus every subset of $m$ edges in $H'$ induces an $\MM$-circuit. This means that $C$ is an $\MM$-circuit, and hence by the choice of $m$, $C$ is a star. In particular, $H'$ has a component of size at least $m+1$. 
    
    By the same reasoning, every forest of size $m$ in $H'$ is a star. This implies that $H'$ is a star, for if it contained a pair of disjoint edges, then we could augment it to a forest in $H'$ of size $m$ that is not a star. (It is easy to verify that apart from vertex stars, every graph on at least four edges has a pair of disjoint edges.)

    This shows that $\psi$ maps vertex stars to vertex stars. It follows by \cref{lemma:star} that $\psi$ is induced by a graph isomorphism, as desired.
\end{proof}

\subsection{Reconstructibility and matroid unions}\label{subsection:unionreconstructibility}

Using the results of the previous subsections, we can give a quick proof of \cref{theorem:unionWhitney}.
We will need the following recent result.

\begin{theorem}\cite[Theorem 1.6]{garamvolgyi.etal_2024a}\label{theorem:dconnpacking}
    For every pair of positive integers $k$ and $c$, there exists a positive integer $c'$ such that every $c'$-connected graph contains $k$ edge-disjoint $c$-connected spanning subgraphs.
\end{theorem}

In fact, \cref{theorem:dconnpacking} is proved in~\cite{garamvolgyi.etal_2024a} by showing, using our terminology, that the $k$-fold union of the family of $d$-dimensional generic rigidity matroids has the Lovász-Yemini property. %

\begin{proof}[Proof of \cref{theorem:unionWhitney}]
    \emph{(a)}     By assumption, there exists a constant $c$ such that every $c$-connected graph is $\MM_i$-rigid for all $i \in \{1,\ldots,k\}$. It follows from \cref{theorem:dconnpacking} that there exist a constant $c'$ such that every $c'$-connected graph contains $k$ edge-disjoint spanning $c$-connected subgraphs. This also means that every $c'$-connected graph contains edge-disjoint $\MM_{i}$-rigid spanning subgraphs for $i \in \{1,\ldots,k\}$, and hence by \cref{lemma:unionthresholdbound}(a) that every $c'$-connected graph is $\MM$-rigid.

    \emph{(b)}
    First, let us suppose that there is an unbounded graph matroid family among $\MM_1,\ldots,\MM_k$, say $\MM_1$. 
    By \cref{theorem:main} and \cref{lemma:boundedLY}, each of $\MM_1,\ldots,\MM_k$ has the Lovász-Yemini property, and hence by %
    part \emph{(a)}, so does $\MM$. Moreover, since $\MM_1$ is unbounded, so is $\MM$. It follows from \cref{theorem:main} that $\MM$ has the Whitney property.

    Thus, we may assume that each of $\MM_1,\ldots,\MM_k$ is bounded. By \cref{theorem:boundedwhitney} and \cref{lemma:smallcircuits}, there exist positive integers $m_1,\ldots,m_k$ such that $K_{1,m_i}$ is the unique (small) $\MM_i$-circuit with at most $m_i$ edges that is a forest, for $i \in \{1,\ldots,k\}$. Let $m = \sum_{i =1}^k(m_{i}-1) + 1$. It is easy to verify now that $K_{1,m}$ is a small $\MM$-circuit. It follows from \cref{theorem:boundedwhitney} that $\MM$ has the Whitney property. 
\end{proof}

We note, without details, that for a nontrivial graph matroid family of dimensionality at least two, the de-coning operation (introduced at the end of \cref{subsection:operations}) also preserves the Lovász-Yemini and the Whitney properties.

\subsection{Extendability}\label{subsection:extendable}

Let us recall the following pair of definitions from the introduction. Given a graph $G$ and a positive integer $d$, the \emph{$d$-dimensional edge split operation} replaces an edge $uv$ of $G$ with
a new vertex joined to $u$ and $v$, as well as to $d-1$ other vertices of $G$. 
Let $\MM$ be a nontrivial, unbounded graph matroid family with dimensionality $d \geq 1$. We say that $\MM$ is \emph{$1$-extendable} if the $d$-dimensional edge split operation preserves $\MM$-independence. 

Nguyen defined $1$-extendability in the context of abstract rigidity matroids~\cite{nguyen_2010}. The naming comes from the fact that the $d$-dimensional edge-split operation is sometimes also called the $d$-dimensional $1$-extension operation.

Our goal in this subsection is to prove \cref{theorem:1extendable}, which says that $1$-extendable graph matroid families have the Lovász-Yemini property. This can be shown by adapting the proof of \cref{theorem:Soma} in~\cite{villanyi_2023}. Since we are not concerned with optimizing the connectivity constant $c$, we instead follow another, more straightforward argument, also due to Villányi~\cite{somaunpublished}. The same proof strategy was used in~\cite{garamvolgyi.etal_2024a} to establish the Lovász-Yemini property for unions of generic rigidity matroids.

Given a graph $G$ and a positive integer $k$, a set of vertices $U \subseteq V(G)$ is a \emph{$k$-dominating set} in $G$ if every vertex $v \in V(G) - U$ has at least $k$ neighbors in $U$ in $G$. It is well-known that high minimum degree forces the existence of small $k$-dominating sets, in the following sense.

\begin{theorem}(See, e.g.,~\cite[Corollary 1]{hansberg.volkmann_2009})\label{theorem:dominating}
    Let $k$ be a positive integer. For every $\varepsilon > 0$ there exists a positive integer $\delta$ such that every graph $G$ with minimum degree at least $\delta$ has a $k$-dominating set $U$ of size at most $\varepsilon |V(G)|$.
\end{theorem}

Let us fix a nontrivial, unbounded graph matroid family $\MM$ with dimensionality $\dimensionality{}$, threshold $\threshold{}$ and rank function $r$ for the remainder of this section. Let $\extendableconstant = r(K_{\threshold{}}) - d \cdot \threshold{}$, so that for every $n \geq \threshold{}$, $r(K_n) = dn + \extendableconstant$. Let $\varepsilon_\MM = \frac{1}{6d}$, and let $\delta_\MM$ be a positive integer such that every graph $G$ with minimum degree at least $\delta_\MM$ has a $\dimensionality{}$-dominating set of size at most $\varepsilon_\MM \cdot |V(G)|$. The existence of $\delta_\MM$ is guaranteed by \cref{theorem:dominating}. We may also assume that $\delta_\MM \geq \max(6\extendableconstant,\threshold{})$. 

The following is the main technical lemma in our proof of \cref{theorem:1extendable}.

\begin{lemma}\label{lemma:removablevertex}
    Let $\MM$ and $\delta_\MM$ be as in the above discussion and let $G = (V,E)$ be a graph. If the minimum degree of $G$ is at least $\delta_\MM$, then $G$ has a vertex $v_0 \in V$ such that $r(G) \leq r(G-v_0) + d$.
\end{lemma}
\begin{proof}
    Let $n$ denote the number of vertices of $G$; note that $n \geq \delta_\MM + 1 > \max(6\extendableconstant,\threshold{}).$ By the definition of $\delta_\MM$, there exists a $\dimensionality{}$-dominating set $U \subseteq V$ in $G$ of size at most $\varepsilon_\MM \cdot n$. Observe that \[\varepsilon_\MM = \frac{1}{6d} = \frac{1}{3d} - \frac{1}{6d} < \frac{1}{3d} - \frac{\extendableconstant}{dn},\]where for the last inequality we used the assumption that $n > 6\extendableconstant$.
    
    Let us construct a subgraph $G_0 = (V,E_0)$ of $G$ by fixing $\dimensionality{}$ edges between $v$ and $U$, for every $v \in V - U$. It follows from \cref{lemma:0extension} that $G_0$ is $\MM$-independent. Thus we may extend $G_0$ to an $\MM$-independent subgraph $G_1 = (V,E_0 \cup E_1)$ of $G$ with $r(G_1) = r(G)$.

    Now \[|E_1| = r(G) - r(G_0) \leq dn + \extendableconstant - d(1 - \varepsilon_\MM)n = r_0 + dn \cdot \varepsilon_\MM < \extendableconstant + (\frac{n}{3} - \extendableconstant) = \frac{n}{3}.\]Hence there are at least $\frac{n}{3} > \varepsilon_\MM \cdot n$ vertices not incident to any edge in $E_1$. In particular, there is such a vertex $v_0 \in V - U$. By construction, this means that $\deg_{G_1}(v_0) = d$. It follows that \[r(G- v_0) \geq r(G_1) - d = r(G) - d,\] as claimed.
\end{proof}

\begin{proof}[Proof of \cref{theorem:1extendable}]
    Let $c = \delta_\MM$, where $\delta_\MM$ is defined as above, and let $G = (V,E)$ be a $c$-connected graph. We prove, by induction on $|V|$, that $G$ is $\MM$-rigid. This shows that $\MM$ has the Lovász-Yemini property with constant $c$. 

    In the base case, $G$ is a complete graph on $c+1$ vertices, which is trivially $\MM$-rigid. Let us thus assume that $|V| \geq c+2$. It follows from \cref{lemma:removablevertex} that $G$ has a vertex $v_0 \in V$ such that $r(G) \leq r(G-v_0) + d$. Note that $\delta_\MM \geq t$, and hence $d_G(v_0) \geq \delta_\MM \geq t \geq d+1$.

    We claim that for every pair of nonadjacent vertices $x,y \in N_G(v_0)$, we have $r(G-v_0) = r(G-v_0+xy)$. Indeed, otherwise we could take an $\MM$-independent subgraph $H$ of $G-v_0+xy$ that contains $xy$ and has $r(G-v_0) + 1$ edges. We can obtain an $\MM$-independent subgraph $H'$ of $G$ with $r(G-v_0) + d +1$ edges by performing a $\dimensionality{}$-dimensional edge split on $x$ and $y$, where the newly added vertex is $v_0$, and using the fact that $\MM$ is $1$-extendable. It follows that $r(G) \geq r(G-v_0) + d + 1$, contradicting the choice of $v_0$.
    
    Let $G' = G-v_0+K_{N_G(v_0)}$. The argument in the previous paragraph shows that $r(G-v_0) = r(G')$. Note that $G'$ is also $c$-connected, since it can be obtained from $G + K_{N_G(v_0)}$ by the deletion of $v_0$, and it is easy to see that deleting a vertex whose neighbors form a clique cannot destroy $c$-connectivity, unless the graph has only $c+1$ vertices. By induction, $G'$ is $\MM$-rigid, and hence so is $G-v_0$. Since $d_G(v_0) \geq c \geq d$, \cref{lemma:0extension} implies that $G$ is also $\MM$-rigid, as desired.
\end{proof}

Among the graph matroid families considered in \cref{subsection:operations}, the family of count matroids $\MM_{k,\ell}$, the families $\MM_{\CC}$ constructed by Schmidt, the family of generic rigidity matroids $\mathcal{R}_d$, and the family $\CC_d$ of generic $C^{d-1}_{d-2}$-cofactor matroids are all $1$-extendable. It is also not difficult to show that the operations of taking unions and de-coning both preserve $1$-extendability.

However, the family $\evencycle$ of even cycle matroids is not $1$-extendable, since applying the $1$-dimensional edge split operation (i.e., edge subdivision) to an odd cycle, which is $\evencycle$-independent, results in an even cycle, which is not. Also note that $\evencycle$ does not have the Lovász-Yemini property, since a bipartite graph on at least three vertices is not $\evencycle$-rigid, no matter how highly vertex-connected it is. It is unclear if there exists a nontrivial, unbounded graph matroid family that is not $1$-extendable, but does have the Lovász-Yemini property.

\section{Concluding remarks}\label{section:concluding}

We conclude by discussing potential extensions of the ideas considered throughout the paper.

\vspace{-.5em}

\paragraph*{Other properties.} Both the Lovász-Yemini and the Whitney properties are assertions about sufficiently highly vertex-connected graphs. 
We may also consider other one-parameter families of graph properties $\mathcal{P}(c)$ besides $c$-connectivity. (Here the range of the parameter $c$ is the natural numbers, and we assume that property $\mathcal{P}(c+1)$ implies property $\mathcal{P}(c)$, for every $c$.) Let us say that a graph matroid family $\MM$ \emph{has the Lovász-Yemini property with respect to $\mathcal{P}$} if there is some constant $c$ such that every graph satisfying $\mathcal{P}(c)$ is $\MM$-rigid. We also define the Whitney property w.r.t.\ $\mathcal{P}$ similarly. \cref{lemma:whitneytoLY} generalizes to this setting as follows. 

\begin{lemma}\label{lemma:whitneytoLYgeneral}
    Let $\mathcal{P}(c)$ be a one-parameter family of graph properties such that $\mathcal{P}(c)$ is edge-monotone for every nonnegative integer $c$. If an unbounded graph matroid family $\MM$ has the Whitney property w.r.t.\ $\mathcal{P}$, then it has the Lovász-Yemini property w.r.t.\ $\mathcal{P}$.
\end{lemma}

\noindent It would be interesting to identify other properties $\mathcal{P} = \mathcal{P}(c)$ for which the converse holds.

\vspace{-.5em}

\paragraph*{Other settings.}

Viewed abstractly, a graph matroid family consists of a collection of matroids defined on some class of combinatorial objects (in this case, edge sets of finite simple graphs) that is invariant under certain maps (in this case, graph isomorphisms), as well as compatible, in the sense that the matroid assigned to a sub-object (in this case, a subgraph) is obtained as a restriction of the original matroid.

This blueprint can be easily adapted to other settings, such as
\begin{itemize}
    \item bipartite graph matroid families, where the objects are edge sets of finite bipartite graphs, and the morphisms are rooted isomorphisms (graph isomorphisms that preserve the color classes);
    \item multigraph matroid families, where the objects are edge sets of finite multigraphs, and the morphisms are graph isomorphisms;
    \item $k$-hypergraph matroid families, where the objects are edge sets of $k$-uniform hypergraphs, and the morphisms are hypergraph isomorphisms.
\end{itemize}

These models can capture additional matroid families appearing in the literature, such as bipartite rigidity matroids~\cite{kalai.etal_2015}, count matroids on hypergraphs~\cite{whiteley_1996}, or low-rank tensor completion matroids~\cite{cruickshank.etal_2023}.
It would be interesting to investigate whether some analogues of \cref{theorem:main,theorem:unionWhitney,theorem:1extendable} hold in these settings. 

In this direction, it was recently shown~\cite{garamvolgyi.etal_2025a} that the family of \emph{generic $(a,b)$-birigidity matroids}, defined on the edge sets of bipartite graphs, has the Lovász-Yemini-property, in the sense that every sufficiently highly connected bipartite graph has the same rank in these matroids as the complete bipartite graph on the same bipartition. More generally, following the methods of~\cite{garamvolgyi.etal_2025a} it can be shown that every bipartite graph matroid family in which the so-called \emph{double $1$-extension} operation preserves independence has the Lovász-Yemini property. We suspect that these families also have the (bipartite version of the) Whitney property. See~\cite{kiraly.etal_2013} for basic results about the structure of bipartite graph matroid families.

We note that $k$-hypergraph matroid families were investigated by Kalai~\cite{kalai_1990} under the name of symmetric $k$-matroids. Given such a family of matroids $\MM$, the \emph{growth function} of $\MM$ is the function mapping each positive integer $n$ to the $\MM$-rank of the complete $k$-uniform hypergraph on $n$ vertices. Kalai gave a characterization of growth functions of $k$-hypergraph matroid families. In particular, he proved a generalization of \cref{lemma:ranklinear}, saying that the growth function of $\MM$ is eventually polynomial, for any $\MM$.

\vspace{-.5em}

\paragraph*{Reconstruction from matroids.}

As noted in the introduction, throughout the paper, the term $\MM$-reconstructibility is used in the sense that the problem of reconstructing $G$ from $\MM(G)$ has a unique solution, up to isomorphism.~\cref{theorem:main,theorem:unionWhitney,theorem:1extendable} provide many examples of graph matroid families $\MM$ for which every sufficiently highly connected graph is $\MM$-reconstructible in this sense. 

Nonetheless, even if $G$ is $\MM$-reconstructible, it is unclear whether we can actually construct $G$ from $\MM(G)$ efficiently (e.g., using a polynomial number of $\MM$-independence oracle calls). Finding such reconstruction algorithms for particular graph matroid families remains an interesting open problem. It seems that such an algorithm is only available for the family of graphic matroids~\cite{seymour_1981}.

\section*{Acknowledgements}

I am grateful to the anonymous referees for their careful reading of the manuscript.
I would also like to thank Tibor Jordán, Csaba Király, and Soma Villányi for useful discussions.

This research was supported by the Lendület Programme of the Hungarian Academy of Sciences (grant number LP2021-1/2021), the National Research, Development and Innovation Fund of Hungary, financed under the ELTE TKP 2021-NKTA-62 funding scheme, and the NKFIH Advanced grant no.\ 152786.

\hypersetup{linkcolor=blue}

\printbibliography

\end{document}